\documentclass[12pt, oneside, notitlepage]{amsart}
\usepackage[margin=3cm]{geometry}
\usepackage{amsmath,amssymb,amsthm,graphicx,mathrsfs,bbm,url}
\usepackage{amsthm}
\usepackage{wrapfig}
\usepackage{enumitem}
\usepackage{mathtools}
\usepackage[utf8]{inputenc}
 \usepackage[T1]{fontenc}
\usepackage[usenames,dvipsnames]{color}
\usepackage[colorlinks=true,linkcolor=Red,citecolor=Green]{hyperref}
\usepackage[super]{nth}
\usepackage[open, openlevel=2, depth=3, atend]{bookmark}
\hypersetup{pdfstartview=XYZ}
\usepackage[font=footnotesize]{caption}
\usepackage{epstopdf}
 
\usepackage{hyperref}

\theoremstyle{plain}
\newtheorem{theorem}{Theorem}[section]
\newtheorem*{theorem*}{Theorem}

\newtheorem{lemma}[theorem]{Lemma}
\newtheorem{proposition}[theorem]{Proposition}
\newtheorem{corollary}[theorem]{Corollary}

\theoremstyle{definition}
\newtheorem{definition}[theorem]{Definition}
\newtheorem{example}[theorem]{Example}

\theoremstyle{remark}
\newtheorem{remark}[theorem]{Remark}

\numberwithin{equation}{section}

\newcommand{\C}{\mathbb{C}}
\newcommand{\R}{\mathbb{R}}

\newcommand{\Z}{\mathbb{Z}}
\newcommand{\cF}{\mathcal{F}}

\newcommand{\X}{\mathbf{X}}

\newcommand{\eps}{\varepsilon}

\newcommand{\mc}{\mathcal}

\newcommand{\dd}{\mathrm{d}}
\newcommand{\leaf}{\mathrm{leaf}}
\newcommand{\Cchart}{\Theta}

\DeclareMathOperator{\vol}{vol}

\DeclareMathOperator{\Op}{Op}
\DeclareMathOperator{\WF}{WF}

\DeclareMathOperator{\supp}{supp}

\DeclareMathOperator{\comp}{comp}

\newcommand{\be}{\begin{equation}}
\newcommand{\ee}{\end{equation}}

% ...

% ...

\title
[Propagation of regularity along unstable manifolds]
{Propagation of regularity along unstable manifolds}

\author{Thibault Lefeuvre}
\address{Universit\'e Paris-Saclay, Laboratoire de math\'ematiques d’Orsay, 91405, Orsay, France.}
\email{thibault.lefeuvre1@universite-paris-saclay.fr}

\author{Rafael Potrie}
\address{Centro de Matem\'atica, Facultad de Ciencias, Universidad de la Rep\'ublica, Montevideo, Uruguay \ and \ IRL-IFUMI2030 CNRS, Laboratorio del Plata.}
\email{rpotrie@cmat.edu.uy}

\begin{document}
%% Résumé

%% Résumé anglais
\begin{abstract}
Let $\varphi_t : M \to M$ be a flow on a smooth closed connected manifold $M$ that preserves and expands a foliation $\mc{F}$. We establish a theorem of propagation of regularity along the leaves of $\mc{F}$ for sections of vector bundles satisfying a transport equation involving the generator of a cocycle over $(\varphi_t)_{t \in \R}$.  As a consequence, we prove a regularity result for Pollicott-Ruelle resonant states: if such state is smooth in restriction to a piece of an unstable leaf, then it is in fact smooth over the entire manifold. We also announce further applications related to joint integrability of extreme bundles of partially hyperbolic diffeomorphisms. The proofs rely on a leafwise semiclassical pseudodifferential calculus adapted to a foliated space, which may be of independent interest.
\end{abstract}

\maketitle

\section{Introduction}

Let $M$ be a smooth closed connected manifold equipped with an arbitrary Riemannian metric $g$. Consider a smooth flow $\varphi_t : M \to M$ generated by a vector field $X$, and assume that there exists a foliation $\cF$ with smooth leaves, continuous transverse structure, and invariant under the flow, such that the leaves of $\cF$ are uniformly expanded. Precise definitions of such foliations are given in \S\ref{ss.deffoliations} (see also \S\ref{sssection:expanded-foliation}), but one may keep in mind the strong unstable foliation of a partially hyperbolic or Anosov flow.

Since the leaves of $\cF$ are smooth, it makes sense to consider a leafwise smooth vector bundle $E \to M$ with a smooth Hermitian (resp. Euclidean) metric along leaves, equipped with a metric connection $\nabla^E$. The vector bundle may be only smooth along leaves of $\cF$ but not necessarily along other directions. See \S\ref{sssection:leafwise-smooth-bundles} for a definition of leafwise smooth vector bundles. This assumption will be discussed in the main results.

Let $\boldsymbol{\varphi}_t : E \to E$ be a fiberwise linear extension of the flow $\varphi_t : M \to M$ to $E$ (i.e. a linear cocycle). We denote by $C^\infty_{\mathrm{leaf}}(M,E)$ the space of \emph{leafwise smooth sections}, that is, sections which are continuous and smooth when restricted to leaves (see \S\ref{sssection:smooth-functions}). The \emph{propagator} is the (semi-)group of operators
\[
e^{-t\X} : C^\infty_{\mathrm{leaf}}(M,E) \to C^\infty_{\mathrm{leaf}}(M,E), \quad e^{-t\X} u (x) = \boldsymbol{\varphi}_t(u(\varphi_{-t}x)) \in E_x, \forall u \in C^\infty_{\mathrm{leaf}}(M,E).
\]
The operator
\[
\X : C^\infty_{\mathrm{leaf}}(M,E) \to C^\infty_{\mathrm{leaf}}(M,E)
\]
is a differential operator of order $1$ called the generator (of the propagator), satisfying the Leibniz rule:
\[
\X(f \otimes u) = X f \otimes u + f \otimes \X u, \qquad \forall f \in C^\infty_{\mathrm{leaf}}(M), u \in C^\infty_{\mathrm{leaf}}(M,E).
\]
We are interested in understanding the regularity properties of sections of $E$ satisfying a transport equation. For instance, if $u  \in C^0(M,E)$ is a section such that $e^{t\X}u=u$ for all $t \in \R$, we want to see under which conditions we can improve its regularity. These conditions are typically phrased as follows: if the restriction of $u$ to an open set $U$ of a leaf is smooth, then $u$ is globally leafwise smooth. More generally, one can study the regularity of distributional sections $u \in \mc{D}'(M,E)$ satisfying a wavefront set condition and solution to the transport equation
\begin{equation}
\label{equation:transport}
(\X-\lambda) u = 0, \qquad \lambda \in \C.
\end{equation}
We show that there exists a threshold regularity $N_0 := N_0(\X,\lambda)$ such that the following holds: if $u$ is $H^N$ in some piece of leaf for some $N > N_0$, then $u$ is globally leafwise $H^N$. See Theorem \ref{corollary:main}.

\subsection{Propagation of regularity} 

Let us state our main result and unfold its applications. Given $x \in M$, let $W^u(x) \in \mc{F}$ be the leaf passing through $x$. Let $W^u_\eps(x)$ denote the $\eps$-ball in the leaf (computed with respect to the metric induced by $g$ on leaves). We also define the \emph{positive flow-out} of the unstable $\eps$-neighborhood ($\eps > 0$) of $x$ by:
\begin{equation}
\label{equation:omegaplus}
\Omega_+(x,\eps) :=  \bigcup_{t \geq \nu_0 |\log(\eps)|} \varphi_t(W^u_{\eps/2}(x)),
\end{equation}
where the constant $\nu_0 > 0$ is chosen such that for $t \geq \nu_0 |\log(\eps)|$, $\varphi_t(W^u_{\eps}(x))$ contains $W^u_{1}(\varphi_t y)$ for all $y \in W^u_{\eps/2}(x)$. We also introduce
\[
\Lambda_+(x,\eps)= \bigcap_{T>0} \overline{ \bigcup_{t \geq T} \varphi_t(W^u_{\eps/2}(x))}.
\]
Additionally, we need to introduce two exponents. We let $\lambda_1/2 > 0$ be the maximal exponential rate of growth of the propagator $(e^{-t\X})_{t \in \R}$ on $L^2$ in restriction to a leaf, where the maximum runs over all possible leaves $L \in \mc{F}$ (see \eqref{equation:growth1} for a precise definition), and $\lambda_2 > 0$ be the minimal expansion rate of the (co)tangent bundle to the foliation (see \eqref{equation:growth2}). Finally, define
\begin{equation}
\label{equation:threshold}
\boldsymbol{\mu} := \dfrac{\lambda_1}{4\lambda_2}.
\end{equation}

Let us state our main results:

\begin{theorem}
\label{corollary:main}
Assume that $E \to M$ is a leafwise smooth vector bundle. Let $\lambda \in \C$, $\gamma > 0$ and let $N \geq 0$ be an integer such that
\begin{equation}
\label{equation:regularity-threshold}
N > \boldsymbol{\mu} + \Re(\lambda)/2\lambda_2.
\end{equation}
There exist constants $C := C(\X, N, \gamma) > 0$ and $\nu := \nu(\X,N) > 0, \nu' := \nu'(\X,N)>0$, such that the following holds. Let $u \in C^0(M,E)$ be a section such that 
$e^{t\X} u = e^{\lambda t} u$ for all $t \in \R$, and assume there exists $x_0 \in M$, $\eps > 0$ such that $u|_{W^u_\eps(x_0)} \in H^{2N}(W^u_\eps(x_0))$.

Then for all $x = \varphi_t y \in \Omega_+(x_0,\eps)$ (where $y \in W^u_{\eps/2}(x_0)$), $u|_{W^{u}_1(x)} \in H^{2N}(W^u_1(x))$ and the following inequality holds: 
\begin{equation}
\label{equation:bound-main-corollary}
\|u\|_{H^{2N}(W^u_1(x))} \leq C \left(\|u\|_{C^0(M,E)} + \eps^{-\nu'} e^{-\nu t} \|u\|_{H^{2N}(W^u_\eps(x_0))}\right).
\end{equation}
In addition, for all $x \in \Lambda_+(x_0,\eps)$, $u|_{W^{u}_1(x)} \in H^{2N}(W^u_1(x))$ and 
\begin{equation}
\label{equation:bound-main2-corollary}
\|u\|_{H^{2N}(W^u_1(x))} \leq C\|u\|_{C^0(M,E)}.
\end{equation}

\end{theorem}

See \S\ref{ssection:sobolev-norms} for the definition of the Sobolev norms. The constant $C$ in \eqref{equation:bound-main2-corollary} is the same as the one in \eqref{equation:bound-main-corollary}. The restriction on $t \geq \nu_0 |\log(\eps)|$ in the definition of $\Omega_+(x,\eps)$ (see \eqref{equation:omegaplus}) stems from the fact that the norm on the left-hand side of \eqref{equation:bound-main} is computed on a ball of size $1$. The assumption being that $u$ is $H^{2N}$ on an $\eps$-neighborhood of an unstable leaf, one needs to propagate for a time $t \gg 1$ large enough so that this neighborhood becomes of size $1$. We emphasize that the constant $C > 0$ in \eqref{equation:bound-main-corollary} and \eqref{equation:bound-main2-corollary} is explicit (see the proof in \S\ref{ssection:proof-propagation-estimates}). In particular, it depends \emph{uniformly} on the flow, provided the vector field is perturbed in the $C^{r}$-topology for some large (explicit) $r>0$ depending on $N$.

We also prove a similar statement to Theorem \ref{corollary:main} for distributions, under an additional wavefront set condition. For this result, one needs to assume that $E \to M$ is a globally smooth vector bundle. In the following statement, $\mc{D}'_\Gamma(M,E)$ denotes the space of distributions with wavefront set contained in $\Gamma$ (see \S\ref{ssection:convention-distributions} for a brief reminder on the theory of distributions). We let $d$ be the dimension of the leaves of $\mc{F}$. Below, $\|u\|_{\Gamma,N,n}$ is a certain semi-norm on $\mc{D}'_\Gamma(M,E)$ (see Lemma \ref{lemma:restriction} for further details).

\begin{theorem}
\label{theorem:main}
Assume that $E \to M$ is a globally smooth vector bundle. Let $\Gamma \subset T^*M \setminus \{0\}$ be a closed conic subset such that
\[
\Gamma \cap N^*\mc{F} = \emptyset.
\]
Let $\lambda \in \C$, $\gamma > 0$ and let $N \geq 0$ be an integer such that
\[
N > \boldsymbol{\mu} + \Re(\lambda)/2\lambda_2.
\]
There exist constants $C := C(\X, \Gamma, N, \gamma) > 0$ and $\nu := \nu(\X,N)>0, \nu':=\nu'(\X,N)>0$ such that the following holds. Let $u \in \mc{D}'_\Gamma(M,E) \cap H^{-2N+(n-d)/2+\gamma}(M,E)$ be a section such that $e^{t\X} u = e^{\lambda t} u$, and assume there exists $x_0 \in M$, $\eps > 0$ such that 
\[
u|_{W^u_\eps(x_0)} \in H^{2N}(W^u_\eps(x_0)).
\]

Then for all $x = \varphi_t y \in \Omega_+(x_0,\eps)$ (where $y \in W^u_{\eps/2}(x)$), one has $u|_{W^{u}_1(x)} \in H^{2N}(W^u_1(x))$ and the following inequality holds: 
\begin{equation}
\label{equation:bound-main}
\|u\|_{H^{2N}(W^u_1(x))} \leq C \left(\|u\|_{H^{-2N+(n-d)/2+\gamma}(M)} + \|u\|_{\Gamma,N,n}+ \eps^{-\nu'} e^{-\nu t} \|u\|_{H^{2N}(W^u_\eps(x_0))}\right).
\end{equation}
In addition, for all $x \in \Lambda_+(x_0,\eps)$, one has $u|_{W^{u}_1(x)} \in H^{2N}(W^u_1(x))$ and 
\begin{equation}
\label{equation:bound-main2}
\|u\|_{H^{2N}(W^u_1(x))} \leq C\left(\|u\|_{H^{-2N+(n-d)/2+\gamma}(M)}+ \|u\|_{\Gamma,N,n}\right).
\end{equation}
\end{theorem}

Both results have similar statements and similar proofs. However, since the assumptions are somewhat different, we state them separately. The latter result is better suited to the study of Pollicott–Ruelle resonances. One could envision a unified statement by imposing weaker assumptions on transverse regularity than the wavefront set condition. We leave such extensions to the interested reader, as we are not aware of any potential applications of a stronger statement.
%Finally, let us mention that dual estimates can be formulated for a foliation $\mc{F}$ with \emph{uniformly contracted} leaves; in this case, the Sobolev spaces $H^{2N}$ of positive regularity are replaced by negative Sobolev spaces $H^{-2N}$.

\subsection{Global smoothness}

We say that the foliation $\mc{F}$ is \emph{minimal} if each leaf of $\mc{F}$ is dense in $M$. As an immediate application of Theorem \ref{theorem:main}, we obtain a result asserting that a section satisfying a transport equation and a wavefront set condition is smooth provided it is smooth on a small piece of a leaf of $\mc{F}$: 

\begin{corollary}
\label{corollary:utile}
Assume that $E \to M$ is a globally smooth vector bundle and that the foliation is $\mc{F}$ is minimal. Let $\Gamma \subset T^*M \setminus \{0\}$ be a closed conic subset such that
\[
\Gamma \cap N^*\mc{F} = \emptyset.
\]
Let $\lambda \in \C$, $u \in \mc{D}'_\Gamma(M,E)$ be a section such that $e^{t\X} u = e^{\lambda t} u$. Assume that
\[
u|_{W^u_\eps(x)} \in C^\infty(W^u_{\eps}(x),E)
\]
for some $x \in M$, $\eps > 0$. Then $u \in C^\infty_{\mathrm{leaf}}(M,E)$. In addition, if $\mc{F}$ is absolutely continuous, then $u \in C^\infty(M,E)$.
\end{corollary}

See Definition \ref{definition:foliation} for the notion of absolute continuity. This result applies to partially hyperbolic flows with minimal strong unstable foliation; many examples of such flows are known. One can also obtain results without this minimality assumption, but in that case the conclusion yields smoothness only on a minimal subset of the foliation. Note that a similar statement can be obtained by applying Theorem \ref{corollary:main} to obtain propagation of smoothness under less transverse regularity to get leafwise regularity.

\subsection{Application to Pollicott-Ruelle resonant states} We now further assume that $\varphi_t : M \to M$ is a smooth Anosov flow with continuous splitting
 \[
 TM = E_0 \oplus E_s \oplus E_u, \qquad E_0 := \R X, \qquad X := \partial_t \varphi_t|_{t = 0}.
 \]
In this case, we let $\mc{F} := W^u$ be the unstable foliation. Generalized Pollicott-Ruelle resonant states are distributions $u \in \mc{D}'_{E_s^\perp}(M,E)$ with wavefront set in the subbundle $E_s^\perp \subset T^*M$ defined by $E_s^\perp(E_0 \oplus E_s) = 0$, solutions to the transport equation $(\X-\lambda)^\ell u = 0$, where $\lambda \in \C$ and $\ell \geq 1$. When $\ell=1$, $u$ is called a Pollicott-Ruelle resonant state. These distributions appear naturally in the study of decay of correlations for flows, see \S\ref{ssection:resonances} for a background discussion.

Due to the condition on their wavefront set, these distributions have a well-defined restriction to the strong unstable leaves $u|_{W^u(x)}$ for any $x \in M$. Similarly to Corollary \ref{corollary:utile}, the following result follows from Theorem \ref{theorem:main}:

 \begin{corollary}
 \label{corollary:smoothness} 
Assume that $E \to M$ is a globally smooth vector bundle and $\varphi_t : M \to M$ is transitive. Let $u \in \mc{D}'_{E_s^\perp}(M,E)$ be a Pollicott-Ruelle resonant state associated with the resonance $\lambda \in \C$. Suppose that there exist $x \in M$ and $V \subset W^u(x)$, an open subset of a strong unstable leaf, such that $u|_V \in C^\infty(V)$. Then $u \in C^\infty(M,E)$.  \end{corollary}

Corollary \ref{corollary:smoothness} does not follow formally from Corollary \ref{corollary:utile} applied with $\mc{F}=W^u$ (the strong unstable foliation) as the flow may be transitive without $\mc{F}$ being minimal; however, this phenomenon can only happen for suspensions with a constant roof function.

It is unlikely that Corollary \ref{corollary:smoothness} holds for generalized resonant states. If $(\X-\lambda)u^\ell = 0$ for $\ell \geq 2$, one further needs to assume that $\X^k u|_{V}$ is smooth for some $V \subset W^u(x)$ for all $0 \leq k \leq \ell -1$. Then, combining Theorem \ref{theorem:main} with Duhamel's formula (see the proof of \cite[Lemma 7.11]{Zworski-12} for instance), it should follow that $u \in C^\infty(M,E)$.
 
In addition, our method also allows to improve a result of Weich \cite{Weich-17} who showed that generalized Pollicott-Ruelle resonant states of Anosov flows have full support in $M$ under the assumption the flow is transitive (that is, if they vanish on an open set of $M$, they are identically zero). We establish that for (non-generalized) resonant states, it suffices that they vanish on a piece of unstable leaf to deduce that they vanish everywhere.

\begin{corollary}
\label{corollary:support}
Assume that $E \to M$ is a globally smooth vector bundle and that $\varphi_t : M \to M$ is transitive. Let $u$ be a Pollicott-Ruelle resonant state whose restriction $u|_V$ to an open subset $V \subset W^u(x)$ of a strong unstable leaf vanishes. Then $u \equiv 0$. In particular, Pollicott-Ruelle resonant states have full support.
\end{corollary}

Similarly to Corollary \ref{corollary:smoothness}, this should also hold for generalized resonant states under stronger assumptions (namely $\X^k u|_{V} \equiv 0$ for all $0 \leq k \leq \ell-1$). We also emphasize that our method of proof allows to recover easily the result of Weich \cite{Weich-17}: a generalized resonant state vanishes on an open subset if and only if it is identically zero. 

Finally, let us point out that a similar result should also hold for $h$-dependent quasimodes as $h \to 0$ (that is, solutions to the high-frequency problem $(\X-\lambda_h) u_h = \mc{O}(h^\ell)$ where $\Im(\lambda_h)=h^{-1}$). The study of such quasimodes is a key step in the proof of exponential mixing for Anosov flows, see \cite{Cekic-Guillarmou-21} or \cite[Chapter 4]{Cekic-Lefeuvre-24}. We believe that the approach developed in the present paper (showing that the restriction of quasimodes to unstable leaves is not smooth) could help. To some extent, this is reminiscent of the ideas developed in \cite{Leclerc-25} showing Fourier decay of equilibrium states of volume-preserving Axiom A diffeomorphisms on surfaces under the assumption that the stable bundle is not $C^2$.

\subsection{Non joint integrability}

One important motivation for this work is to extend measure rigidity results in dimension $3$ to higher dimensions (see \cite{ALOS,Katz,EskinPotrieZhang}). In these articles, a strategy introduced in \cite{tsujii.expvol,tsujii-zhang} is used to prove similar results on propagation of regularity that rely very strongly on normal forms; our motivation for the present article is to be able to avoid the normal forms as they are not continuous in higher dimensions. In a forthcoming work with Elliot Smith, we expect to use her thesis work \cite{ElliotSmith} together with Theorem \ref{corollary:main} to deduce measure rigidity statements via \cite{BrownEskinFilipRHertz}.

\subsection{Leafwise semiclassical pseudodifferential calculus} The main tool for establishing Theorem \ref{theorem:main} is the introduction of an appropriate algebra of semiclassical pseudodifferential operators adapted to the foliation. We believe that this calculus is of independent interest and could be used for other purposes. Let us briefly explain how it works. Let $\mc{F}$ be a foliation of $M$; we further assume that the leaves are smooth and vary continuously in transverse directions (see Definition \ref{definition:foliation}). The space $C^\infty_{\mathrm{leaf}}(M)$ denotes the space of continuous functions over $M$ whose restriction to each leaf $L \in \mc{F}$ of the foliation is uniformly smooth (that is the restriction to each leaf is smooth, and for all $k \geq 0$, the $C^k$-norm on the leaf varies continuously in transverse directions), see \S\ref{sssection:smooth-functions}.

The algebra $\Psi^\bullet_{h,\mathrm{leaf}}(M,\mc{F})$ is an algebra of operators $\mathbf{A} : C^\infty_{\mathrm{leaf}}(M) \to C^\infty_{\mathrm{leaf}}(M)$ with the property that they restrict to each leaf $L \in \mc{F}$ and induce an operator $\mathbf{A}_L : C^\infty_{\comp}(L) \to C^\infty_{\comp}(L)$ on compactly supported smooth functions on the leaf $L$ which is a semiclassical operator in the usual sense \cite{Zworski-12}. To avoid issues due to the noncompactness of the leaves, we further impose that the kernel of $\mathbf{A}_L$ is supported in a small open subset of the diagonal, and that the size of this off-diagonal support is uniform with respect to all leaves $L \in \mc{F}$. We then establish all the standard results available in the theory of semiclassical pseudodifferential operators for this new class of operators such as: the existence of a principal symbol, the algebra property, a Calderon-Vaillancourt-type theorem, Egorov's theorem, the sharp Gårding inequality, etc.

\subsection{Comparison with other work} The estimate \eqref{equation:bound-main} stated in Theorem \ref{theorem:main} is reminiscent of radial source estimates in the context of hyperbolic dynamics, see \cite{Dyatlov-Zworski-16,Bonthonneau-Lefeuvre-23}. These estimates are microlocal in the sense that they are stated using pseudodifferential operators microlocalized near $E_s^\perp$ in $T^*M$; however, they also exhibit a threshold condition, similar to \eqref{equation:regularity-threshold}. They can be used in place of the classical Journé lemma \cite{Journe-86} to bootstrap the regularity of solutions to cohomological equations, see \cite{Bonthonneau-Lefeuvre-23} for further discussion. Our approach is also reminiscent of the one initiated in \cite{Rauch-Taylor-05} aiming at generalizing the Journé lemma using microlocal analysis. See also \cite[Section 8]{Fisher-Kalinin-Spatzier-13} for a similar discussion.

The idea of considering a foliated pseudodifferential calculus is not new and already appears across the literature, see \cite{Connes-82,Androulidakis-Skandalis-10, Kordyukov-91,Kordyukov-02} among other references. However, to the best of our knowledge, it seems to be new in the context of (partially) hyperbolic dynamics. We believe that it could have other applications as well.

\subsection{Organization of the paper}  The paper is organized as follows:
\begin{itemize}
\item In \S\ref{section:leafwise-calculus}, we discuss the notion of transversally continuous foliations with smooth leaves, introduce the leafwise semiclassical pseudodifferential calculus, and prove its main properties.
\item In \S\ref{section:preliminaries}, we discuss several analytic properties in the context of a foliated manifold: the growth of propagators for flows preserving the foliation, the Fubini theorem for distributions, etc. We also recall the notion of Pollicott-Ruelle resonant states.
\item The main results (Theorems \ref{corollary:main} and \ref{theorem:main}) are then proved in \S\ref{section:proofs}. We also prove the statements on the global smoothness and Pollicott-Ruelle resonant states in \S\ref{ssection:pr}.
\end{itemize}

\subsection{Conventions for distributions} \label{ssection:convention-distributions}
Let $\Omega^1 M \to M$ be the density bundle over $M$. Note that $\Omega^1M$ is isomorphic to the vector bundle of volume forms $\Lambda^n T^*M$ if $M$ is orientable. The space $\mc{D}'(M)$ of distributions on $M$ is the topological dual of $C^\infty(M,\Omega^1M)$, the space of smooth densities (that is, sections of this bundle). By fixing a density $\omega$ in $M$, one gets a natural identification of $C^\infty(M,\Omega^1M)$ with $C^\infty(M)$. There is a natural embedding $C^\infty(M) \hookrightarrow \mc{D}'(M)$ given by
\[
(u,\varphi) := \int_M u(x) \varphi(x), \qquad \forall \varphi \in C^\infty(M,\Omega^1M). 
\]
That is, $\mc{D}'(M)$ should be thought of the \emph{generalized functions} on $M$.

Given $u \in \mc{D}'(M)$, its wavefront set $\WF(u) \subset T^*M \setminus \{0\}$ is defined as the \emph{complement} of the set of points $(x_0,\xi_0)$ such that there exists an open conic neighborhood $V_{\xi_0} \subset T_{x_0}^*M$ of $\xi_0$ and a function $\chi \in C^\infty(M)$ with $\chi(x_0) \neq 0$, such that for all $S \in C^\infty(M)$ with $dS(x_0) \in V_{\xi_0}$ one has that $(\chi u, e^{-iS/h}) = \mc{O}(h^\infty)$ (i.e. it decays faster than any polynomial in $h$ when $h \to 0$). We refer to \cite[Chapter 4]{Lefeuvre-book} for details. The wavefront set is always a closed conic subset of $T^*M \setminus \{0\}$. Given a closed conic subset $\Gamma \subset T^*M \setminus \{0\}$, we let $\mc{D}'_{\Gamma}(M)$ denote the space of distributions $u$ such that $\WF(u) \subset \Gamma$. This space can be equipped with a continuous family of seminorms, see \cite[Remark 4.1.10]{Lefeuvre-book} for instance.

Finally, for $x \in \R^n$, $\langle x \rangle := \sqrt{1+|x|^2}$ denotes the Japanese bracket of $x$, where $|\bullet|$ is the Euclidean norm in $\R^n$.

\medskip

\noindent\textbf{Data Availability Statement:} This article is theoretical in nature. No external data were used; all results are derived within the paper.

\medskip

\noindent\textbf{Conflict of Interest Statement:} The authors declare no conflict of interest.

\medskip

\noindent\textbf{Acknowledgement:} This project was supported by the European Research Council (ERC) under the European Union’s Horizon 2020 research and innovation programme (Grant agreement no. 101162990 — ADG) and finished while the second author was in residence in SLMath, Berkeley. R.P. was partially supported by CSIC . We thank S. Crovisier, A. Eskin, R. Elliot Smith, S. Filip, D. Fisher, S. Mu\~noz-Thon, T. Weich and Z. Zhang for comments and discussions. We also thank the anonymous referee for their comments which hopefully improved the presentation of the manuscript.

\section{Leafwise pseudodifferential operators}\label{section:leafwise-calculus}

Our aim is to introduce families of pseudodifferential operators on a transversally continuous foliation $\mc{F}$ with smooth leaves. Since the leaves are non-compact, special care is required when manipulating these operators. This motivates the introduction of a \emph{leafwise uniform} semiclassical pseudodifferential calculus. Roughly speaking, an operator $\mathbf{A} \in \Psi^m_{h,\mathrm{leaf}}(M,\mc{F})$ belongs to this calculus if $\mathbf{A}$ induces on each leaf $L \in \mc{F}$ a (smooth) semiclassical pseudodifferential operator, and $\mathbf{A}$ depends continuously on the leaf $L(x)$ as $x$ varies transversally to it. We assume the reader is familiar with the standard semiclassical and microlocal calculus in $\R^n$ and on manifolds, see \cite{Zworski-12, Lefeuvre-book} for references. 

The section is organized as follows:
\begin{itemize}
\item In \S\ref{ssection:assumptions}, we define the transversally continuous foliations with smooth leaves on which the calculus is defined;
\item In \S\ref{ssection:rn}, we briefly detail the leafwise semiclassical calculus in $\R^n$ and discuss the change of coordinates;
\item In \S\ref{ssection:calculus-manifold}, we define the leafwise calculus on closed manifolds, and detail its main properties;
\item Finally, we introduce the leafwise Sobolev spaces in \S\ref{ssection:sobolev-norms}.
\end{itemize}

\subsection{Foliations}

\label{ssection:assumptions}

We first introduce the class of foliations $\mc{F}$ of $M$ we will be working with. The reader should bear in mind the strong unstable foliation of a partially hyperbolic flow as an example. Recall that $n=\dim(M)$ and $d$ is the dimension of the leaves of $\mc{F}$.

\subsubsection{Definition}\label{ss.deffoliations} Our aim is to introduce a class of foliations with smooth leaves, such that the leaves vary ``continuously'' transversally to the foliation. This is the purpose of the following definition, mainly inspired by \cite{HirschPughShub,DeLaLlave-01} (see the notion of $C^{0,\infty+}$ foliations in \cite{Candel-Conlon-00,Candel-Conlon-03}):

\begin{definition}[Transversally continuous foliation with smooth leaves] \label{definition:foliation}
A \emph{transversally continuous foliation with smooth leaves} is a partition of
\[
M = \bigsqcup_{L \in \mc{F}} L,
\]
into smooth $d$-dimensional complete leaves $L$ (where $\mc{F}$ is called the \emph{space of leaves}) such that the following holds. For all $p \in M$, the exists a small neighborhood $U \subset M$ of $p$ and a continuous map 
\begin{equation}
\label{equation:admissible-coordinates}
\Cchart : V_{1} \times V_{2} \to U,
\end{equation}
where $V_{1} \subset \R^{n-d}, V_{2} \subset \R^{d}$ are connected open subsets, such that:
\begin{itemize}
\item For all $x_1 \in V_{1}$, $\{\Cchart(x_1,x_2) ~:~ x_2 \in V_{2}\} = L \cap U$ is the connected component of a leaf $L$ in $U$;
\item For all $\alpha \in \Z_{\geq 0}^{d}$, $\partial^\alpha_{x_2} \Cchart$ is continuous in $V_{1} \times V_{2}$.
\end{itemize}
In addition, the foliation is \emph{absolutely continuous} if:
\begin{itemize}
\item There exists a continuous function $J \in C^0(V_{1} \times V_{2})$ such that for all measurable subsets $E_{1} \subset V_{1}, E_{2} \subset V_{2}$:
\begin{equation}
\label{equation:utile}
\vol(\Cchart(E_{1} \times E_{2})) = \int_{E_{1} \times E_{2}} J(x_1,x_2) \dd x_1 \dd x_2;
\end{equation}
\item For all $\alpha  \in \Z_{\geq 0}^{d}$, $\partial^\alpha_{x_2} J$ is continuous on $V_{1} \times V_{2}$.
\end{itemize}
\end{definition}

Here, $\vol$ denotes the Riemannian volume on $M$ induced by the arbitrary background metric $g$. A similar definition can be given in finite regularity (that is with $C^r$-regular leaves and/or higher transversal regularity), see the discussion around \cite[Theorem 2]{DeLaLlave-01}). The coordinate mappings $\Cchart$ satisfying the above properties as in \eqref{equation:admissible-coordinates} are called \emph{admissible} coordinates (or admissible maps or charts). It is a non-trivial theorem that the unstable (or the stable) foliation of a partially hyperbolic flow is a transversally continuous foliation with smooth leaves; in addition, it is also absolutely continuous—see \cite[Theorem 2]{DeLaLlave-01} for the Anosov case and the discussion in \cite[\S 2.2]{burns-wilkinson} for the partially hyperbolic case.

 We will use the generic letter $L$ for leaves of $\mc{F}$. Later, when the leaves will be assumed to be preserved and expanded by a flow $\varphi_t : M \to M$, we will denote them by $W^u$.

\subsubsection{Leafwise diffeomorphisms} \label{sssection:leafwise-diffeo}

Suppose that $(M',\mc{F}')$ and $(M,\mc{F})$ are two smooth manifolds (not necessarily closed) equipped with transversally continuous foliations with smooth leaves (Definition \ref{definition:foliation}).

\begin{definition}[Leafwise diffeomorphisms]
\label{definition:leafwise-diffeo}
A homeomorphism $\kappa : M' \to M$ is a \emph{leafwise diffeomorphism} if:
\begin{itemize}
\item For all leaves $L' \in \mc{F}'$, there exists a leaf $L \in \mc{F}$ such that
\[
\kappa|_{L'} : L' \to L
\]
is a smooth diffeomorphism;
\item For all admissible charts $\Cchart, \Cchart'$ on $M$ and $M'$ respectively (see \eqref{equation:admissible-coordinates}), for all $\alpha \in \Z^{d}_{\geq 0}$, the homeomorphism
\[
\kappa_{\Cchart,\Cchart'} : (\Cchart')^{-1}(\kappa^{-1}(U) \cap U')  \to \Cchart^{-1}(U \cap \kappa(U')), \qquad \kappa_{\Cchart,\Cchart'} := \Cchart^{-1} \circ \kappa \circ \Cchart',
\]
satisfies that $\partial^\alpha_{x_2}\kappa_{\Cchart,\Cchart'}$ is continuous on $(\Cchart')^{-1}(\kappa^{-1}(U) \cap U') \subset V'_{1} \times V'_{2}$. 
\end{itemize}
\end{definition}

Notice that the admissible chart $\Cchart : V_{1} \times V_{2} \to U$ defined in \eqref{equation:admissible-coordinates} is an example of a leafwise diffeomorphism in the above sense. It is straightforward to check that the homeomorphism $\kappa_{\Cchart,\Cchart'}$ is of the form
\begin{equation}
\label{equation:ecriture}
\kappa_{\Cchart,\Cchart'}(x'_1,x'_2) = (\psi(x'_1), \phi(x'_1,x'_2)),
\end{equation}
where $\psi : V'_{1} \to V_{1}$ is a homeomorphism and $\phi : V'_{1} \times V'_{2} \to V_{2}$ is smooth with respect to the $x'_2$-variable and continuous with respect to $x'_1$.

\subsubsection{Tangent and cotangent bundles} Let $T\mc{F} \subset TM$ be the subbundle of $TM$ defined at a point $x \in M$ by $T_x\mc{F} := T_x L$, where $L$ is the unique leaf of the foliation $\mc{F}$ passing through $x \in M$. In the admissible coordinates \eqref{equation:admissible-coordinates}, $T\mc{F}$ is spanned by the vectors $\{ \partial_{x_2}^1, ..., \partial_{x_2}^d\}$. Notice that, in general, there is no natural complement of $T\mc{F}$ in $TM$ (see Remark \ref{remark:chat} below).

The conormal bundle $N^*\mc{F} \subset T^*M$ is the subbundle of $T^*M$ defined at $x \in M$ by
\[
N^*_x \mc{F} := \{\xi \in T^*_xM ~:~ (\xi,v) = 0, \forall v \in T_xL\}.
\]
In the admissible coordinates \eqref{equation:admissible-coordinates}, writing $x_1 = (x_{1,1}, x_{1,2}, ..., x_{1,n-d})$, $N^*\mc{F}$ is spanned by $\{\dd x_{1,1}, ..., \dd x_{1,n-d}\}$. Unless further assumptions are made, there is no natural complement of $N^*\mc{F}$ in $T^*M$. The dual bundle to $T\mc{F}$ (the vector bundle of linear forms on $T\mc{F}$) is denoted by $T^*\mc{F}$. Observe that there is a tautological identification:
\begin{equation}
\label{equation:identification}
T^*\mc{F} \simeq T^*M/N^*\mc{F}, \qquad T^*M/N^*\mc{F} \ni [\xi] \mapsto \xi|_{T\mc{F}} \in T^*\mc{F}.
\end{equation}
This map is well-defined, that is independent of the choice of representative $\xi$ in the class $[\xi]$ because two representatives differ by an element which vanish on $T\mc{F}$. However, we emphasize that, in general, $T^*\mc{F}$ cannot be seen as a natural subbundle of $T^*M$.

Notice that the vector bundle $T^*\mc{F}$ can be defined alternatively in the admissible coordinates \eqref{equation:admissible-coordinates} as the vector bundle whose transition matrices are given by
\begin{equation}
\label{equation:matrices}
V'_1 \times V'_2 \times \R^d \ni (x'_1,x'_2,\xi) \mapsto (\kappa(x'_1,x'_2), \partial_{x'_2}\phi(x'_1,x'_2)^{-\top}\xi) \in V_1 \times V_2 \times \R^d,
\end{equation}
where $\phi$ is defined in \eqref{equation:ecriture}. In particular, a function $a \in C^0(T^*\mc{F})$ on $T^*\mc{F}$ is equivalent to the data of a function $a_\Cchart \in C^0(V_1\times V_2 \times \R^d)$ in each admissible coordinate patch \eqref{equation:admissible-coordinates} satisfying the equivariance property:
\begin{equation}
\label{equation:equivariance}
a_{\Cchart}(\kappa(x'), \partial_{x_2}\phi(x'_1,x'_2)^{-\top}\xi') = a_{\Cchart'}(x',\xi'), \qquad \forall x' \in V'_1 \times V'_2, \xi \in \R^d.
\end{equation}

Finally, any leafwise diffeomorphism $\kappa : M' \to M$ induces an action $\dd\kappa : T\mc{F}' \to T\mc{F}$ given by the action of the differential on tangent vectors to the leaves. As a consequence, there is a well-defined dual action:
\[
\dd\kappa^{\top} : T^*\mc{F} \to T^*\mc{F}', \qquad (\dd\kappa^{\top}(x)\xi, v) := (\xi, \dd\kappa(x)v),
\]
where $x \in M', v \in T_x\mc{F}'$ and $\xi \in T^*_{\kappa(x)}\mc{F}$.

\begin{remark} \label{remark:chat} When $T\mc{F}$ is complemented and there exists a globally defined (continuous) subbundle $G \subset TM$ such that $TM = T\mc{F} \oplus G$, then $T^*\mc{F}$ can be naturally identified with the conormal bundle $N^*G \subset T^*M$ to $G$ (hence with a subbundle of $T^*M$) defined as
\[
N^*G := \{ \xi \in T^*M ~:~ (\xi,v) = 0, \forall v \in G\}.
\]
In particular, if $\mc{F} = W^u$ is the strong unstable foliation of an Anosov flow, then $T\mc{F} = E_u$ is naturally complemented by $G= E_0 \oplus E_s$ and the remark applies.
\end{remark}

\subsubsection{Leafwise smooth functions} \label{sssection:smooth-functions} We now discuss the notion of leafwise smooth functions (or sections) on $M$.

\begin{definition}[Leafwise smooth functions]
\label{definition:leafwise-smooth-functions}
Let $f \in C^0(M)$. We say that $f$ is \emph{leafwise smooth} if for all adapted charts $\Cchart$ as in \eqref{equation:admissible-coordinates}, for all $\alpha \in \Z_{\geq 0}^{d}$, $\partial_{x_2}^\alpha (\Cchart^*f)$ is continuous on $V_1 \times V_2$. The set of leafwise smooth functions is denoted by $C^\infty_{\mathrm{leaf}}(M)$.
\end{definition}

The notion of leafwise smooth functions on $T^*\mc{F}$ is also well-defined. Namely $a \in C^\infty_{\leaf}(T^*\mc{F})$ if and only if, in each admissible coordinate system \eqref{equation:admissible-coordinates}, the corresponding function $a_{\Cchart}$ on $V_1\times V_2 \times \R^d$ satisfies that for all $\alpha, \beta \in \Z_{\geq 0}^d$ the function $\partial_\xi^\beta \partial_{x_2}^\alpha a_{\Cchart}$ is continuous on $V_1\times V_2 \times \R^d$. Using \eqref{equation:equivariance}, it is immediate to verify that this property is invariant by admissible changes of coordinates.

The restriction of the metric $g$ on $M$ to each leaf $L$ of $\mc{F}$ is a metric $g_L$ with Levi-Civita connection $\nabla_{\mc{F}}$. It defines an operator 
\begin{equation}
\label{equation:leafwise-derivative}
\nabla_{\mc{F}} : C^\infty_{\leaf}(M) \to C^\infty_{\leaf}(M,T^*\mc{F}),
\end{equation}
such that $(\nabla_{\mc{F}})f(x ; v) := ((\nabla_{\mc{F}})_v f) (x)$ for all $f \in C_{\leaf}^\infty(M), x \in M, v \in T_x\mc{F}$. We also recall that $\nabla^k_{\mc{F}} f (x) \in \mathrm{Sym}^k T^*_x\mc{F}$ is the symmetric tensor such that 
\[
\nabla^k_{\mc{F}} f (x; v, ..., v) = \partial_t^k(f(\gamma(t)))|_{t = 0},
\]
where $\gamma : (-\eps, \eps) \to L(x)$ is the geodesic generated by the metric restricted to the leaf $L(x)$.

Given a leaf $L \subset M$, we denote by $C^\infty(L)$ the space of smooth uniformly bounded functions on the leaf:
\begin{equation}
\label{equation:cl}
C^\infty(L) := \left\{f : L \to \C ~:~ \forall k \geq 0, \sup_{x \in L} |\nabla_{\mc{F}}^k f(x)| < \infty\right\}.
\end{equation}
Observe that, for all leaves $L \in \mc{F}$, there is a well-defined restriction operator:
\begin{equation}
\label{equation:restriction}
\mathbf{r}_L : C^\infty_{\leaf}(M) \to C^\infty(L), \qquad \mathbf{r}_Lf := f|_L.
\end{equation}
That $\mathbf{r}_L f \in C^\infty(L)$ (in the sense of \eqref{equation:cl}) is left as a verification for the reader. We also emphasize that the map $\mathbf{r}_L$ may not be surjective (if $L$ is non-compact for instance), not even on $C^\infty_{\comp}(L)$, the space of smooth compactly supported functions on $L$.

\subsubsection{Leafwise smooth vector bundles}

\label{sssection:leafwise-smooth-bundles} Let $E \to M$ be a complex topological vector bundle over $M$ of rank $k \geq 1$.

\begin{definition}[Leafwise smooth vector bundles]
The vector bundle $E \to M$ is \emph{leafwise smooth} if for every contractible pair of neighborhoods $U_1, U_2 \subset M$, there exists a trivializing map
\[
\varphi_i : U_i \times \C^k \to \pi^{-1}(U_i), \qquad i=1,2,
\]
such that the composite function
\[
\varphi_1^{-1} \circ \varphi_2 : (U_1 \cap U_2) \times \C^k \to (U_1 \cap U_2) \times \C^k
\]
is well-defined on the overlap and satisfies $\varphi_1^{-1} \circ \varphi_2(x,s) = (x,g(x)s)$, where $g$ is a $\mathrm{GL}_k(\C)$-valued function, and $g$ is leafwise smoth in the sense of Definition \ref{definition:leafwise-smooth-functions} (that is, all entries of the matrix are leafwise smooth).
\end{definition}

The same definition remains valid by replacing $\C$ by any other field. In this paper, it will be used only for real and complex vector bundles. For instance, the strong unstable bundle $E_u \to M$ of an Anosov flow is leafwise smooth for the foliation $\mc{F} = W^u$ by strong unstable leaves.

\begin{example}
\label{example:leafwise-bundles}
If $\mc{F}$ is a transversally continuous foliation with smooth leaves, then $T\mc{F}$ and $T^*\mc{F}$ are leafwise smooth vector bundles. The proof is immediate.
\end{example}

If $E \to M$ is a leafwise smooth vector bundle over $M$, there is similarly a well-defined notion of leafwise smooth sections $f \in C^\infty_{\mathrm{leaf}}(M,E)$ with values in $E$. For that, it suffices to consider small enough admissible charts $\Cchart : V_1 \times V_2 \to U$ such that $E$ is trivial in $U$; it is then required that each coordinate of the section $f$ is continuous and leafwise smooth in the above sense. Note that, since admissible charts are only continuous in the $x_1$-coordinate, having that $E \to M$ is leafwise smooth means precisely that there exist such charts.

Finally, the following holds:

\begin{lemma}
\label{lemma:wavefront-leafwise}
Suppose that $\mc{F}$ is absolutely continuous. Let $f \in C^\infty_{\mathrm{leaf}}(M,E)$. Then $\WF(f) \subset N^*\mc{F}$.
\end{lemma}

\begin{proof}
See \cite[Lemma 1.11]{Bonthonneau-Guillarmou-Weich-24}. 
\end{proof}

\subsubsection{Leafwise volume forms}

\label{sssection:leafwise-volume}

Recall that the background metric $g$, when restricted to each leaf $L \in \mc{F}$, defines a metric denoted by $g_L$, with smooth Riemannian density $\vol_L$. Given $x \in M$, and $y \in L(x)$, we introduce the notation $\mu_x(y) := \vol_{L(x)}(y)$; note that $\mu_x \in C^\infty(L(x), \Omega^1 L(x))$. In general, this cannot be seen as a distribution over $M$ due to the noncompactness of the leaves. However, letting $\chi_x \in C^\infty_{\comp}(L(x))$ be the function $\chi_x(y) := \chi(d_L(x,y))$, where $\chi \in C^\infty_{\comp}(\R_+)$ is a smooth bump function equal to $1$ on $[0,1)$ and $0$ outside of $[0,2)$, and $d_L$ denotes the distance induced by $g_L$ on $L(x)$, we can see $\chi_x \mu_x \in \mc{D}'(M,\Omega^1 M)$ as a distributional density on $M$ (that is the dual of $C^\infty(M)$) defined by
\begin{equation}
\label{equation:mux}
(\chi_x \mu_x, \varphi) := \int_{L(x)} \varphi(y) \chi_x(y) \dd\mu_x(y).
\end{equation}
The following holds:

\begin{lemma}
\label{lemma:leafwise-density}
For all $x \in M$, the wavefront set of $\chi_x \mu_x$ is contained in $N^*\mc{F}$, that is $\WF(\chi_x \mu_x) \subset N^*\mc{F}$. Additionally, the map
\[
M \to \mc{D}'_{N^*\mc{F}}(M,\Omega^1 M), \qquad x \mapsto \chi_x \mu_x
\]
is continuous.
\end{lemma}

\begin{proof}
The claim on the wavefront set is immediate, see \cite[Example 4.3.7]{Lefeuvre-book} for instance. In a patch of admissible coordinates $\Theta$, writing $x=(x_1,x_2)$, one has for $\varphi$ a smooth function with compact support in the chart:
\begin{equation}
\label{equation:pen}
(\chi_x \mu_x, \varphi) = \int_{\R^d} \varphi(x_1,y_2) \chi(d_{L(x_1)}(x_2,y_2)) \alpha(x_1,y_2) \dd y_2,
\end{equation}
where $\alpha \in C^0(V_1, C^\infty(V_2))$ is smooth with respect to the second variable, and $\mu_x = \mu_{x_1} = \alpha \cdot \dd y_2$. (Indeed, the regularity on $\alpha$ follows from the fact that the metric $g_L$ in smooth in restriction to each leaf $L$, and transversally continuous.) If $x_n \to x$, then one sees from \eqref{equation:pen} that $(\chi_{x_n} \mu_{x_n}, \varphi) \to (\chi_x \mu_x, \varphi)$, which proves that $M \to \mc{D}'(M,\Omega^1 M)$, $x \mapsto \chi_x \mu_x$ is continuous.

Finally, the continuity as a map $M \to \mc{D}'_{N^*\mc{F}}(M,\Omega^1 M)$ follows by considering seminorms on $\mc{D}'_{N^*\mc{F}}(M,\Omega^1 M)$ (see \cite[Definition 4.1.8]{Lefeuvre-book} and \cite[Remark 4.1.10]{Lefeuvre-book} for a definition of those). Indeed, pick $x_\star \in M$ and a local \emph{smooth} chart $\Theta : U \to V \subset \R^n$ around $x_\star$. Let $\pi : T^*V \simeq V \times \R^n_\xi \to \R^n_\xi$ be the projection onto the second factor. Let $C \subset \R^n_\xi$ be a closed cone not intersecting $\pi(N^*\mc{F})$ on $V$. Then for all $\xi \in C$, one has:
\begin{equation}
\label{equation:ipp}
\widehat{\chi_x \mu_x}(\xi) = (\chi_x \mu_x,  e^{-i \xi \cdot \bullet}) =   \int_{L_x} e^{-i \xi \cdot y} \chi(d_{L_x}(x,y)) \dd \vol_{L(x)}(y).
\end{equation}
Let $Y_1, ..., Y_d$ be an independent family of leafwise smooth vector fields tangent to the foliation $\mc{F}$ on $V$. Since $\xi \in C$, $\xi \neq 0$, one has $\xi(Y_j) \neq 0$ for some $j \in \{1,...,N\}$. In fact, if $\eps > 0$ is chosen small enough, then $\xi$ is contained in the union of cones $\cup_j \{|\xi(Y_j)| > \eps|\xi|\}$. We can then use the identity
\[
e^{-i \xi \cdot y} = i \xi(Y_j(y))^{-1} Y_j e^{-i \xi \cdot y} = L_j(\xi) e^{-i \xi \cdot y} = L_j(\xi)^N e^{-i \xi \cdot y}
\]
for $N \geq 1$ arbitrary in \eqref{equation:ipp}. Integrating by parts in \eqref{equation:ipp}, we obtain that for all $N \geq 0$, there exists a constant $C_N > 0$ such that for all $x$ near $x_\star$, $|\widehat{\chi_x \mu_x}(\xi)| \leq C_N \langle\xi\rangle^{-N}$. This proves the claim.
\end{proof}

\subsection{Leafwise calculus in $\R^n$}

\label{ssection:rn}

We now introduce a leafwise semiclassical calculus in $\R^n$. Let $V_1 \subset \R^{n-d}, V_2 \subset \R^d$ be two open connected subsets. We use the coordinates $(x_1,x_2,\xi_2)$ to denote a generic point in $V_1 \times V_2 \times \R^d$.

\subsubsection{Definition} The symbol class $S^m_{h,\mathrm{leaf}}(V_1 \times V_2 \times \R^d)$ is defined as the set of all continuous $h$-dependent functions $a \in C^0(V_1 \times V_2 \times \R^d)$ satisfying the following estimates: for all compact subset $K \subset V_1 \times V_2$, for all $\alpha, \beta \in \Z^d_{\geq 0}$, there exist constants $C,h_0 > 0$ such that 
\begin{equation}
\label{equation:bounds-symbol}
|\partial^\beta_{x_2} \partial^\alpha_{\xi} a(h;x_1,x_2,\xi)| \leq C\langle\xi \rangle^{m-|\alpha|}, \qquad \forall h \in (0,h_0) , (x_1,x_2) \in K, \xi \in \R^d.
\end{equation}
There is a slight abuse of notation here, as the set $S^m_{h,\mathrm{leaf}}$ does not depend itself on $h$; only its elements are allowed to depend on $h$. Observe that this class is preserved by the action of leafwise diffeomorphisms (see Definition \ref{definition:leafwise-diffeo})
\[
\kappa : V'_1 \times V'_2 \to V_1 \times V_2.
\]
Namely if $a \in S^m_{h,\leaf}(V_1\times V_2\times \R^d)$, using that $\kappa$ has the form \eqref{equation:ecriture}, we find that the symbol (here $x'=(x'_1,x'_2)$)
\begin{equation}
\label{equation:pullback-symbol}
\widetilde{\kappa}^*a(h;x',\xi') := a(h; \kappa(x'), \partial_{x'_2} \phi^{-\top}(x'_1,x'_2)\xi')
\end{equation}
satisfies $\widetilde{\kappa}^*a \in S^m_{h,\leaf}(V'_1\times V'_2\times \R^d)$. Here $\widetilde{\kappa}$ denotes the induced action on covectors; for simplicity, we will drop the $\widetilde{\bullet}$ in the notation from now on.

The leafwise quantization of $a$ is then defined as
\begin{equation}
\label{equation:leafwise-rn}
\Op_{h,\mathrm{leaf}}^{\R^n}(a)f(x_1,x_2) = \dfrac{1}{(2\pi h)^d} \int_{V_2} \int_{\R^d_{\xi}} e^{\tfrac{i}{h}\xi\cdot(x_2-y_2)} a(x_1,x_2,\xi) f(x_1,y_2) \dd y_2 \dd \xi,
\end{equation}
where $f \in C^\infty_{\comp}(V_1 \times V_2)$. Notice that on each leaf $\{x_1=\mathrm{cst}\}$, the operator $\Op_{h,\mathrm{leaf}}(a)$ restricts to a usual semiclassical pseudodifferential operator with full symbol $a(x_1,\bullet,\bullet)$.

Verifying that the leafwise quantization \eqref{equation:leafwise-rn} in $\R^n$ satisfies the usual properties of a semiclassical calculus such as composition, existence of a pseudodifferential adjoint, uniform boundedness on $L^2$ of operators of order $0$, etc, is exactly the same as in the standard case. Indeed, \eqref{equation:leafwise-rn} is merely a parametrized version (with continuous parameter $x_1 \in V_1$) of the usual quantization in $\R^d$ in the $x_2$-variable. See \cite{Zworski-12} or \cite[Appendix E]{Dyatlov-Zworski-19} for further details on the standard semiclassical calculus. 

\subsubsection{Pullback by an admissible diffeomorphism} We now discuss the pullback of a leafwise pseudodifferential operator by a leafwise diffeomorphism. Let $\kappa : V'_1 \times V'_2 \to V_1 \times V_2$ be a leafwise diffeomorphism; it is of the form:
\[
\kappa(x'_1,x_2) = (\psi(x'_1), \phi(x'_1,x'_2)).
\]
Let $\mathbf{A}' = \Op^{\R^n}_{h,\leaf}(a')$ be a leafwise pseudodifferential operator on $V'_1 \times V'_2$ such that $a' \in S^m_{h,\leaf}$ and $a'(h;\bullet,\bullet,\xi)$ is (uniformly in $\xi \in \R^d$) supported in $V'_1 \times V'_2$, and set:
\[
\mathbf{A} := (\kappa^{-1})^* \circ \mathbf{A}' \circ \kappa^*.
\]
We claim that the following holds:

\begin{lemma}
\label{lemma:egorov-rn}
 Then, the operator $\mathbf{A}$ is of the form $\mathbf{A} = \Op^{\R^n}_{h,\leaf}(a)$ where
\[
a(h;x,\xi) = a'(h;\kappa^{-1}(x), \partial_{x'_2}\phi^{\top}(x')\xi) + \mc{O}_{S^{m-1}_{h,\leaf}}(h),
\]
and $x' = \kappa^{-1}(x)$.
\end{lemma}

Here, $\mc{O}_{S^{m-1}_{h,\leaf}}(h)$ means that the error is of the form $h b$ where $b \in S^{m-1}_{h,\leaf}$.

\begin{proof}
The proof is a straightforward adaptation of the standard argument for semiclassical pseudodifferential operators depending on the variable \( x_2 \in \mathbb{R}^d \) (see for instance \cite[Lemma 5.2.7]{Lefeuvre-book} in the case $h=1$, or \cite[Proposition E.10]{Dyatlov-Zworski-19}), now extended to the case where the operators also depend on the additional parameter \( x_1 \in \mathbb{R}^{n-d} \). 
\end{proof}

\subsection{Leafwise calculus on manifolds}

\label{ssection:calculus-manifold}

Throughout this section, $\mc{F}$ is a transversally continuous foliation on $M$ with smooth leaves, i.e., satisfying the assumptions of \S\ref{ssection:assumptions}. Unless explicitly mentioned, it does not need to be absolutely continuous.

\subsubsection{Definition}  We first introduce the notion of residual operators in the leafwise calculus:

\begin{definition}[Residual operators]
\label{definition:residual}
Let
\[
\mathbf{R} : C^\infty_{\mathrm{leaf}}(M) \to C^\infty_{\mathrm{leaf}}(M)
\]
be an $h$-dependent family of operators. It is called \emph{residual} if:
\begin{itemize}
\item The operator $\mathbf{R}$ can be written as
\begin{equation}
\label{equation:formula-r}
\mathbf{R} f(x) = \int_{L(x)} K_h(x,y) f(y) \dd \vol(y),
\end{equation}
for all $x \in M$, $f \in C^\infty(M)$, where $K_h(x,\bullet) \in C^\infty_{\comp}(L(x))$ is a smooth function with compact support $\supp(K_h(x,\bullet)) \subset \{ y \in L(x) ~:~ d_{L(x)}(x,y) \leq C \}$, $C > 0$ is a uniform constant (independent of $x \in M$ and $h \in (0,h_0)$) and $d_{L(x)}$ denotes the induced Riemannian distance on the leaf.

\item In addition, the following boundedness properties hold. For all $k, k' \geq 0$, 
\[
\sup_{x \in M} \sup_{y \in L(x)} |\nabla^{k'}_{x, \mc{F}} \nabla^k_{y, \mc{F}} K_h(x,y)| = \mc{O}(h^\infty).
\]
\end{itemize}
\end{definition}

We recall that the notation $\mc{O}(h^\infty)$ indicates that the term is bounded by $C_N h^N$ for all $N \geq 0$ and small enough $h$. The operator $\nabla^k_{y, \mc{F}}$ denotes the leafwise derivative defined in \eqref{equation:leafwise-derivative} acting only on the $y$-variable, $x$ being fixed. 

The first property implies that $\mathbf{R}$ restricts to an operator $\mathbf{R}_L$ on each leaf $L$ which is properly supported. The second property implies that $\mathbf{R}$ is a leafwise smoothing operator with norm $\mc{O}(h^\infty)$ when acting on any reasonable functional space. This comment will be made more precise in \S\ref{ssection:sobolev-norms}. 

We can now introduce the class of leafwise $h$-semiclassical pseudodifferential operators relative to $\mc{F}$:

\begin{definition}[Leafwise uniform semiclassical pseudodifferential operators]
\label{definition:leafwise-pseudo}
Let $m \in \R$. An operator
\[
\mathbf{A} : C^\infty_{\mathrm{leaf}}(M) \to C^\infty_{\mathrm{leaf}}(M)
\]
belongs to the class $\Psi^m_{h,\mathrm{leaf}}(M,\mc{F})$ of leafwise semiclassical pseudodifferential operator of order $m$ if the following holds:
\begin{itemize}
\item For all admissible charts $\Cchart$ (as in \eqref{equation:admissible-coordinates}), for all $\chi,\chi' \in C^\infty_{\comp}(U)$, the operator
\[
\Cchart^*\circ \chi \mathbf{A} \chi' \circ (\Cchart^{-1})^* : C^\infty_{\mathrm{leaf}}(V_1 \times V_2) \to C^\infty_{\mathrm{leaf}}(V_1 \times V_2)
\]
is a leafwise semiclassical pseudodifferential operator on $V_1 \times V_2$ in the sense of \S\ref{ssection:rn}, that is $\Cchart^*\circ \chi \mathbf{A} \chi' \circ (\Cchart^{-1})^* = \Op_{h,\mathrm{leaf}}^{\R^n}(a)$ for some symbol $a \in S^m_{h,\mathrm{leaf}}(V_1 \times  V_2 \times \R^d)$;
\item For all $\chi, \chi' \in C^\infty(M)$ such that $\supp(\chi) \cap \supp(\chi') = \emptyset$, $\chi \mathbf{A} \chi'$ is residual in the sense of Definition \ref{definition:residual}.
\end{itemize}
\end{definition}

The same remark as above applies: as a set, $\Psi^m_{h,\mathrm{leaf}}(M,\mc{F})$ does not depend on $h$; only its elements do. The previous definition should be compared with \cite[Proposition E.13]{Dyatlov-Zworski-19} (standard definition of semiclassical pseudodifferential operators without any foliations involved). Recall that $\mathbf{r}_L$ denotes the restriction operator to the leaf $L \in \mc{F}$, see \eqref{equation:restriction}. The following holds:

\begin{lemma}[Induced operator on the leaves]
\label{lemma:induced-operator}
Let $\mathbf{A} \in \Psi^m_{h,\mathrm{leaf}}(M,\mc{F})$. Then for all $L \in \mc{F}$, there exists a properly supported operator $\mathbf{A}_L : C^\infty(L) \to C^\infty(L)$ such that
\[
\mathbf{A}_L \circ \mathbf{r}_L = \mathbf{r}_L \circ \mathbf{A}.
\]
\end{lemma}

In addition, the operator $\mathbf{A}_L$ is $h$-semiclassical pseudodifferential in the usual sense on $L$ and the principal symbol of $\mathbf{A}_L$ is given by $\sigma_{\mathbf{A}}|_L$, the restriction of the principal of $\mathbf{A}$ to $L$ (see \eqref{equation:symbol} below for a definition of the principal symbol).

\begin{proof}
First, observe that the statement is immediate for residual operators, by \eqref{equation:formula-r}. Now, consider a cover of $M = \cup_{i=1}^N U_i$ by open subsets such that there exists local admissible coordinates $\Cchart_i$ mapping onto $U_i$. Let $\mathbf{1}=\sum_i \chi_i$ be a partition of unity subordinated to this cover. Let $\chi_i' \in C^\infty_{\comp}(U_i)$ such that $\chi_i' = 1$ on the support of $\chi_i$. We then write
\[
\mathbf{A} = \sum_{i,j} \chi'_i \chi_j \mathbf{A} \chi_i + (1-\chi'_i) \chi_j \mathbf{A} \chi_i.
\]
It suffices to verify that each term in the sum satisfies the conclusions of the lemma (i.e. induces a properly supported leafwise operator). By Definition \ref{definition:leafwise-pseudo} (second item), $(1-\chi'_i) \chi_j \mathbf{A} \chi_i$ is residual so it induces a leafwise operator by the preliminary observation. As to $\chi'_i \chi_j \mathbf{A} \chi_i$, it is of the form $\Op_{h,\leaf}^{\R^n}(a_{i,j})$ on the coordinate patch $\Cchart_i^{-1}(U_i) \subset \R^n$ and the statement is immediate for such operators using the formula \eqref{equation:leafwise-rn} (see the remark below \eqref{equation:leafwise-rn}).
\end{proof}

\begin{remark}[Leafwise support of operators] \label{remark:support}
Lemma \ref{lemma:induced-operator} shows that the operators in $\Psi^m_{h,\leaf}(M,\mc{F})$ induce properly supported semiclassical pseudodifferential operators on each leaf. In addition, on each leaf, the off-diagonal size of the support of the Schwartz kernel is uniformly bounded (with respect to all leaves). This restriction makes it easier to discuss numerous properties of pseudodifferential operators (such as boundedness on $L^2$) since they can all be reduced to the classical case of $\R^d$, with an additional parameter $x_1 \in \R^{n-d}$ ranging over a compact set. However, this support condition excludes many natural operators. For example, the leafwise inverse $(1+h^2\boldsymbol{\Delta}_{\mc{F}})^{-1}$ of the leafwise Laplacian (see \S\ref{ssection:examples} below where it is introduced) is \emph{not} properly supported. As a consequence, while the leafwise Laplacian $1+h^2\boldsymbol{\Delta}_{\mc{F}}$ lies in $\Psi^2_{h,\leaf}(M,\mc{F})$, and admits a well-defined inverse $(1+h^2\boldsymbol{\Delta}_{\mc{F}})^{-1}$, this inverse \emph{does not} belong to $\Psi^{-2}_{h,\leaf}(M,\mc{F})$ due to the failure of the support property. For most applications, this is not a serious limitation: indeed, it is sufficient to work with a parametrix $\mathbf{B} \in \Psi^{-2}_{h,\leaf}(M,\mc{F})$ such that $\mathbf{B}(1+h^2\boldsymbol{\Delta}_{\mc{F}})-\mathbbm{1}$ is residual.

If the leafwise calculus were defined by allowing residual operators that are not (uniformly) properly supported in the leaves, one would need to impose additional off-diagonal decay conditions on the Schwartz kernel "at infinity in the leaves" in order to establish standard properties, such as (uniform) $L^2$-boundedness of these operators on each leaf. To avoid having to address such issues, we require that the Schwartz kernel of residual operators be uniformly properly supported on each leaf.
\end{remark}

Finally, let us also mention that if $E, E' \to M$ are two leafwise smooth auxiliary vector bundles over $M$, it is possible to define similarly leafwise semiclassical pseudodifferential operators
\[
\mathbf{A} : C^\infty_{\comp}(M,E) \to C^\infty_{\comp}(M,E').
\]
We denote by $\Psi^m_{h,\mathrm{leaf}}(M,\mc{F}, E \to E')$ the class of such operators. The details of this construction are left for the reader.

\subsubsection{Quantization} We let $S^m_{h,\mathrm{leaf}}(T^*\mc{F})$ be the set of leafwise smooth functions $a$ on $T^*\mc{F}$ (see \S\ref{sssection:smooth-functions} where this notion is defined) satisfying \eqref{equation:bounds-symbol} in admissible coordinates. Namely $a \in S^m_{h,\mathrm{leaf}}(T^*\mc{F})$ if and only if in any admissible patch of coordinates \eqref{equation:admissible-coordinates}, the corresponding function $a_{\Cchart} \in C^\infty_{\leaf}(V_1 \times V_2 \times \R^d)$ satisfies the estimates \eqref{equation:bounds-symbol}. Notice that this is intrinsically defined on the manifold $M$ by the change of variable formula \eqref{equation:pullback-symbol}.

We let $M = \cup_{i=1}^N U_i$ be a cover of $M$ by open subsets such that there exists local admissible coordinates $\Cchart_i : V_{i,x_1} \times V_{i,x_2} \to U_i$ as in \eqref{equation:admissible-coordinates}. Consider a partition of unity $\mathbf{1} = \sum_i \chi_i$ subordinated to this cover and functions $\chi'_i \in C^\infty_{\comp}(U_i)$ such that $\chi'_i \equiv 1$ on $\supp(\chi_i)$. For a symbol $a \in S^m_{h,\mathrm{leaf}}(T^*\mc{F})$, we define its leafwise quantization by:
\begin{equation}
\label{equation:quantization-manifold}
\Op_{h,\leaf}(a) := \sum_i \chi'_i (\Cchart_i^{-1})^* \Op_{h,\leaf}^{\R^n}(\Cchart_i^*(\chi_i a)) \Cchart_i^* \chi'_i.
\end{equation}
(Compare with \cite[Equation (E.1.38)]{Dyatlov-Zworski-19}.) It follows from an adaptation of standard arguments that $\Op_{h,\leaf}(a) \in \Psi^m_{h,\leaf}(M,\mc{F})$, see the proof of \cite[Proposition E.15]{Dyatlov-Zworski-19} in the classical case.

\begin{remark}Alternatively, one could define a more intrinsic leafwise quantization by:
\[
\begin{split}
\Op_{h,\mathrm{leaf}}(a)f(x) &= \dfrac{1}{(2\pi h)^d} \int_{L(x)} \int_{T^*_x\mc{F}} e^{-\tfrac{i}{h}\xi \cdot (\exp_x^{\mathrm{leaf}})^{-1}(y)} a(x,\xi) f(y) \chi(d_{\mc{F}}(x,y)) \\
& \hspace{9cm} \dd \vol(\xi) \dd \vol(y),
\end{split}
\]
where $\exp^{\leaf}$ is the leafwise exponential map (that is the exponential map induced by the metric on each leaf), $\chi \in C^\infty_{\comp}(\R)$ is a bump function equal to $1$ in a neighborhood of $0$, $d_{\mc{F}}$ denotes the leafwise Riemannian distance, $\dd \vol(y)$ is the Riemannian volume on the leaf (induced by the restriction of the metric to the leaves) and $\dd\vol(\xi)$ denotes the corresponding volume on $T^*_x\mc{F}$.
\end{remark}

Following standard arguments, it is possible to prove that any $\mathbf{A} \in \Psi^m_{h,\mathrm{leaf}}(M,\mc{F})$ can be written in the form
\begin{equation}
\label{equation:decomposition}
\mathbf{A} = \Op_{h,\mathrm{leaf}}(a) + \mathbf{R},
\end{equation}
for some $a \in S^m_{h,\mathrm{leaf}}(T^*\mc{F})$ and some residual operator $\mathbf{R}$, see \cite[Proposition E.16]{Dyatlov-Zworski-19} for a proof in the standard case; the same proof applies \emph{verbatim}. We emphasize that the decomposition \eqref{equation:decomposition} is not unique.

\subsubsection{Main properties} \label{sssection:main-prop} The leafwise semiclassical calculus enjoys similar properties to the standard semiclassical calculus: there is a well-defined principal symbol, the adjoint belongs to the same class, and the leafwise pseudodifferential operators form an algebra, that is composition is well-defined within this class of operators. The principal symbol of $\mathbf{A} \in \Psi^m_{h,\mathrm{leaf}}(M)$ is defined as follows. Consider an admissible patch of coordinates $\Cchart$ mapping onto a neighborhood $U \subset M$. Let $\chi \in C^\infty_{\comp}(U)$ be a smooth bump function equal to $1$ on an open subset of $U$. By Definition \ref{definition:leafwise-pseudo}, we can write $\Cchart^*\circ \chi \mathbf{A} \chi' \circ (\Cchart^{-1})^* = \Op_{h,\mathrm{leaf}}^{\R^n}(a)$ for some symbol $a \in S^m_{h,\leaf}(V_1 \times V_2 \times \R^d)$. We then define the principal symbol $\sigma_{\mathbf{A}} \in S^m_{h,\mathrm{leaf}}(T^*\mc{F})/hS^{m-1}_{h,\mathrm{leaf}}(T^*\mc{F})$ as
\begin{equation}
\label{equation:symbol}
\sigma_{\mathbf{A}} = [\Cchart^* a], \qquad \text{ on } \{\chi=1\},
\end{equation}
where $[\bullet]$ denotes the equivalence class in $S^m_{h,\leaf}(T^*\mc{F})$ modulo $hS^{m-1}{h,\leaf}(T^*\mc{F})$. Here, $\Cchart^*a$ denotes the pullback of the symbol $a$ to $T^*\mc{F}$. The following holds:

\begin{proposition}[Principal symbol, adjoint and composition] \label{proposition:pseudo}
Let $\mathbf{A} \in \Psi^m_{h,\mathrm{leaf}}(M)$. The following properties hold:

\begin{itemize}
\item The principal symbol of $\mathbf{A}$ is well-defined as an element
\[
\sigma_{\mathbf{A}} \in S^m_{h,\mathrm{leaf}}(T^*\mc{F})/hS^{m-1}_{h,\mathrm{leaf}}(T^*\mc{F}).
\]
\item The adjoint $\mathbf{A}^* \in \Psi^m_{h,\mathrm{leaf}}(M)$ belongs to the leafwise calculus and has principal symbol $\sigma_{\mathbf{A}^*} = \overline{\sigma}_{\mathbf{A}}$.

\item For $\mathbf{A} \in \Psi^m_{h,\mathrm{leaf}}(M,\mc{F})$, $\mathbf{B} \in \Psi^{m'}_{h,\mathrm{leaf}}(M,\mc{F})$, one has
\[
\mathbf{A} \circ \mathbf{B} \in \Psi^{m+m'}_{h,\mathrm{leaf}}(M,\mc{F}).
\]
In addition $\sigma_{\mathbf{A} \mathbf{B}} = \sigma_{\mathbf{A}} \sigma_{\mathbf{B}}$.
\end{itemize}
\end{proposition}

We emphasize that, in general, for $\mathbf{A} \in \Psi^m_{h,\leaf}(M,\mc{F})$ and $L \in \mc{F}$, $(\mathbf{A}_L)^* \neq (\mathbf{A}^*)_L$. However, there is always an operator $\mathbf{B} \in \Psi^m_{h,\leaf}(M,\mc{F})$ such that $\mathbf{B}_L = (\mathbf{A}_L)^*$, and the principal symbol of $\mathbf{B}$ satisfies $\sigma_{\mathbf{B}} = \sigma_{\mathbf{A}^*} = \overline{\sigma_{\mathbf{A}}}$. 

\begin{proof}
The first item follows from Lemma \ref{lemma:egorov-rn} combined with \eqref{equation:equivariance}. As to the second and third item, the proof follows \emph{verbatim} the usual proof for the adjoint and the composition of standard semiclassical pseudodifferential operators, see \cite[Proposition E.17]{Dyatlov-Zworski-19}.
\end{proof}

More generally, for a vector-valued operator $\mathbf{A} \in \Psi^m_{h,\mathrm{leaf}}(M,\mc{F},E \to E')$, the principal symbol is well-defined as an element
\[
\sigma_{\mathbf{A}} \in S^m_{h,\mathrm{leaf}}(T^*\mc{F} \otimes \mathrm{Hom}(E,E'))/hS^{m-1}_{h,\mathrm{leaf}}(T^*\mc{F} \otimes \mathrm{Hom}(E,E')).
\]
The principal symbol of the adjoint is then given by
\[
\sigma_{\mathbf{A}^*}(x,\xi) = \sigma_{\mathbf{A}}(x,\xi)^* \in \mathrm{Hom}(E',E),
\]
where the $\bullet^*$ on the right-hand side is understood in the sense of bounded linear operators between Hermitian vector spaces. In the particular case where $E=E'$, we say that the principal symbol is \emph{diagonal} (or that $\mathbf{A}$ is a \emph{principally diagonal} operator) if $\sigma_{\mathbf{A}} = a \cdot \mathbf{1}_E$ for some scalar symbol $a \in S^m_{h,\leaf}(T^*\mc{F})$.

\subsubsection{Examples} \label{ssection:examples} We now provide two examples of leafwise operators in the calculus. We let $E \to M$ be a smooth Hermitian bundle over $M$. Recall that $T^*\mc{F}$ is a leafwise smooth vector bundle (Example \ref{example:leafwise-bundles}).

\begin{itemize}

\item Let $\boldsymbol{\nabla}^E$ be an arbitrary smooth connection on $E$; it can be seen as a differential operator of order $1$:
\[
\boldsymbol{\nabla}^E : C^\infty(M,E) \to C^\infty(M,T^*M \otimes E).
\]
By restricting the target space (that is, by only plugging vector fields tangent to the foliation in the connection), one can define an \emph{leafwise connection}
\[
\boldsymbol{\nabla}^E_{\mc{F}} : C^\infty(M,E) \to C^\infty_{\leaf}(M, T^*\mc{F} \otimes E).
\]
That is, for all $x \in M, v \in T_x \mc{F}$ and $f \in C^\infty(M,E)$, $\boldsymbol{\nabla}^E_{\mc{F}}f(x;v) := \nabla^E_v f(x)$. It is immediate to verify that
\[
h \boldsymbol{\nabla}^E_{\mc{F}} \in \Psi^1_{h,\mathrm{leaf}}(M, {\mc{F}}, E \to T^*\mc{F} \otimes E)
\]
is a leafwise uniform semiclassical operator of order $1$. Its principal symbol is
\[
\sigma_{h \boldsymbol{\nabla}^E_{\mc{F}}}(x,\xi) = i \xi \otimes \mathbbm{1}_{E_x}, \qquad \forall (x,\xi) \in T^*\mc{F}.
\]

\item Another example of a leafwise pseudodifferential operator is provided by the leafwise Laplacian
\begin{equation}
\label{equation:laplacien}
\boldsymbol{\Delta}_{\mc{F}}^E : C^\infty_{\mathrm{leaf}}(M,E) \to C^\infty_{\mathrm{leaf}}(M,E),
\end{equation}
defined as $\boldsymbol{\Delta}_{\mc{F}}^E := (\boldsymbol{\nabla}^E_{\mc{F}})^*\boldsymbol{\nabla}^E_{\mc{F}}$, where $(\boldsymbol{\nabla}^E_{\mc{F}})^*$ denotes the (leafwise) adjoint of the leafwise operator. Note that $h^2\boldsymbol{\Delta}^E_{\mc{F}} \in \Psi^2_{h,\mathrm{leaf}}(M,\mc{F})$ with principal symbol
\[
\sigma_{h^2\boldsymbol{\Delta}^E_{\mc{F}}}(x,\xi) = |\xi|^2_g \otimes \mathbbm{1}_{E_x}
\]
(This follows from Proposition \ref{proposition:pseudo}.) When $E = M \times \C$ is the trivial line bundle, we simple write $\boldsymbol{\Delta}_{\mc{F}}$ for the leafwise Laplacian.
\end{itemize}

\subsubsection{Ellipticity and parametrices} As in the standard semiclassical calculus, we say that $\mathbf{A} \in \Psi^m_{h,\leaf}(M,\mc{F})$ is \emph{uniformly elliptic} on $T^*\mc{F}$ if there exists constant $C > 0$ such that for all $h > 0$ small enough:
\[
|\sigma_{\mathbf{A}}(x,\xi)| \geq C \langle\xi\rangle^m, \qquad \forall (x,\xi) \in T^*\mc{F}.
\] 
Notice that this inequality is independent of the choice of representative for the principal symbol $\sigma_{\mathbf{A}}$ (modulo $S^m_{h,\leaf}(T^*\mc{F}))$. It can also be extended to principally scalar operators acting on vector bundles. In particular, the operator $1+h^2\Delta_{\mc{F}}^E$ is elliptic since its principal symbol is given by $(1+ |\xi|^2_g) \otimes \mathbbm{1}_{E_x}$.

As in the standard calculus, uniformly elliptic operators can be inverted modulo residual ones:

\begin{lemma}
\label{lemma:parametrix}
Suppose that $\mathbf{A} \in \Psi^m_{h,\leaf}(M,\mc{F})$ is uniformly elliptic on $T^*\mc{F}$. Then, there exists $\mathbf{B} \in \Psi^{-m}_{h,\leaf}(M,\mc{F})$ such that $\mathbf{A}\mathbf{B}-\mathbbm{1}$ is residual.
\end{lemma}

\begin{proof}
This is a mere adaptation of the usual proof (see e.g. \cite[Theorem 5.3.7]{Lefeuvre-book} or \cite[Proposition E.32]{Dyatlov-Zworski-19}). First, consider a representative $a \in S^m_{h,\leaf}(T^*\mc{F})$ of $\sigma_{\mathbf{A}}$. The ellipticity assumption guarantees that $b := 1/a \in S^{-m}_{h,\leaf}(T^*\mc{F})$. In addition, by construction, $\mathbf{A}\Op_{h,\leaf}(b)-\mathbbm{1} = -h \mathbf{R} \in \Psi^{-1}_{h,\leaf}(M,\mc{F})$. By the Borel summation lemma (see \cite[Lemma 5.3.10]{Lefeuvre-book} for instance), there exists $\mathbf{Q} \in \Psi^0_{h,\leaf}(M,\mc{F})$ with asymptotic expansion $\mathbf{Q} \sim \sum_{k \geq 0} (h\mathbf{R})^k$. (Technically, the off-diagonal support of the Schwartz kernel of $(h\mathbf{R})^k$ grows as $k \to +\infty$ so one should modify this operator by multiplying its Schwartz kernel by a cutoff function supported near the diagonal in each leaf; this does not affect the asymptotic summation process.) It then suffices to observe that
\[
\mathbf{A}\Op_{h,\leaf}(b)\mathbf{Q} = \mathbbm{1} + \mathbf{K},
\]
where $\mathbf{K}$ is residual.
\end{proof}

\subsubsection{Action of diffeomorphisms}

In this paragraph, we assume that $\kappa : M \to M$ is a leafwise diffeomorphism in the sense of \S\ref{sssection:leafwise-diffeo}. It is a direct consequence of Lemma \ref{sssection:leafwise-diffeo} that the following (weak) Egorov-type result holds on manifolds: 

\begin{lemma}[Action of diffeomorphisms in the leafwise calculus]
\label{lemma:egorov-leafwise}
Let $\mathbf{A} \in \Psi^m_{h,\leaf}(M,\mc{F})$. Then $\mathbf{A}_\kappa := \kappa^* \circ \mathbf{A} \circ (\kappa^{-1})^* \in \Psi^m_{h,\leaf}(M,\mc{F})$ and has principal symbol
\[
\sigma_{\mathbf{A}_\kappa}(x,\xi) = \sigma_{\mathbf{A}}(\kappa(x), \dd \kappa(x)^{-\top} \xi).
\]
\end{lemma}

\begin{proof}
Immediate consequence of Lemma \ref{lemma:egorov-rn} combined with Definition \ref{definition:leafwise-pseudo}.
\end{proof}

This lemma easily extends to operators acting on vector-valued sections (in $\Psi^m_{h,\leaf}(M,\mc{F},E)$). It will be applied with $\kappa = \varphi_t$, an Anosov flow and $\mc{F} = W^u$, the unstable foliation (see Lemma \ref{lemma:egorov} below).

\subsubsection{The sharp Gårding inequality} 

The following lemma is a leafwise version of the sharp Gårding inequality. It will play an important role in the proof of Theorem \ref{theorem:main}. 

\begin{lemma}
\label{lemma:garding0}
Let $\mathbf{A}  \in \Psi^m_{h,\leaf}(M,\mc{F},E)$ be a principally diagonal operator such that $\sigma_{\mathbf{A}} = a \cdot \mathbf{1}_{E}$ with $\Re(a) \geq 0$. Then there exists $\mathbf{K} \in \Psi^{m-1}_{h,\leaf}(M,E)$ such that for all $L \in \mc{F}$, for all $u \in C^\infty_{\comp}(L)$:
\begin{equation}
\label{equation:garding-aprem}
\Re \langle \mathbf{A}_L u, u \rangle_{L^2(L,E)} \geq h \langle \mathbf{K}_L u, u \rangle_{L^2(L,E)}.
\end{equation}
In addition, if $\mc{F}$ is absolutely continuous (see Definition \ref{definition:foliation}), then for all $u \in C^\infty(M,E)$: 
\begin{equation}
\label{equation:garding-matin}
\Re \langle \mathbf{A} u, u \rangle_{L^2(M,E)} \geq h \langle \mathbf{K} u, u \rangle_{L^2(M,E)}.
\end{equation}
\end{lemma}

More generally, by continuity, \eqref{equation:garding-aprem} and \eqref{equation:garding-matin} hold for functions (or distributions) in $H^{m/2}_{\mathrm{comp}}(L)$.

\begin{proof}
We prove the lemma in the case where $E$ is the trivial bundle; the extension to vector-valued operators is immediate. By \eqref{equation:decomposition}, we can write $\mathbf{A} = \Op_{h,\leaf}(a) + \mc{O}(h^\infty)$ where the remainder term is a residual operator. Hence, it suffices to prove the claim for $\Op_{h,\leaf}(a)$, assuming $a \geq 0$. By the quantization formula \eqref{equation:quantization-manifold}, it suffices to prove that each term in the sum defining the quantization is non-negative in the sense of quadratic form modulo a $\mc{O}(h)$ remainder. The inequality \eqref{equation:garding-aprem} follows the standard sharp Gårding inequality in $\R^d$, see \cite[Theorem 4.32]{Zworski-12} for a proof.

We now prove \eqref{equation:garding-matin}. This implies verifying that the leafwise quantization \eqref{equation:leafwise-rn} in $\R^n$ satisfies the sharp Gårding inequality, that is
\[
\Re \langle \Op^{\R^n}_{h,\leaf}(b)u,u\rangle_{L^2(\R^n, \mu)} \geq h \langle \mathbf{K} u, u \rangle_{L^2(\R^n,\mu)},
\]
where $\mu$ denotes the pullback of the volume form $\vol$ to $\R^n$ in the admissible coordinate patch $V_1 \times V_2 \subset \R^n$, and $\Re(b) \geq 0$ is a leafwise symbol supported in the patch. By assumption, $\mc{F}$ is absolutely continuous and thus, by \eqref{equation:utile}, $\mu = J(x_1,x_2) \dd x_1 \dd x_2$. Writing $\mathbf{B} :=  \Op^{\R^n}_{h,\leaf}(b)$ and $L(x_1) := \{(x_1,x_2) ~:~ x_2 \in \R^d\}$, we find that:
\begin{equation}
\label{equation:faim0}
\Re \langle \mathbf{B} u,u\rangle_{L^2(\R^n, \mu)} = \int_{\R^{n-d}_{x_1}} \Re \langle  \mathbf{B} _{L(x_1)} u|_{L(x_1)}, u|_{L(x_1)}\rangle_{L^2(\R^d,J(x_1,\bullet)\dd x_2)} \dd x_1
\end{equation}
Observe that this is now merely a parametrized version of the standard sharp Gårding inequality in $\R^d$ with parameter $x_1 \in \R^{n-d}$, see \cite[Theorem 4.32]{Zworski-12} for a proof. We get that there exists $\mathbf{K} \in \Psi^{m-1}_{h,\leaf}$ such that for all $x_1 \in V_1$, for all $u \in C^\infty_{\comp}(V_2)$,
\begin{equation}
\label{equation:faim}
\Re \langle  \mathbf{B} _{L(x_1)} u, u\rangle_{L^2(\R^d,J(x_1,\bullet)\dd x_2)} \geq h \langle\mathbf{K}_{L(x_1)} u, u \rangle_{L^2(\R^d,J(x_1,\bullet)\dd x_2)}.
\end{equation}
Inserting \eqref{equation:faim} into \eqref{equation:faim0} finally proves \eqref{equation:garding-matin}.
\end{proof}

\subsection{Sobolev norms} \label{ssection:sobolev-norms}

We first introduce the leafwise Sobolev spaces on $M$. For simplicity, we only discuss the spaces of even order $\pm 2N$ where $N \in \Z_{\geq 0}$. The theory naturally extends to $N \in \R_{\geq 0}$ by interpolation, but we do not expand this point to avoid cumbersome discussions.

\subsubsection{Definition. First properties} Given $u \in C^\infty(M)$, define:
\[
\|u\|^2_{H^{2N}_{h,\leaf}(M)} := \|(1 + h^2\boldsymbol{\Delta}_{\mc{F}})^N u\|^2_{L^2(M)}.
\]
The space $H^{2N}_{h,\leaf}(M)$ is defined as the completion of $C^\infty(M)$ with respect to the previous norm. Note that $C^\infty_{\leaf}(M) \subset \cap_{N \geq 0} H^{2N}_{1,\leaf}(M)$. However, the reverse inclusion does not hold because the $H^{2N}_{1,\leaf}$-norm does not control transverse regularity to the leaves (and leafwise smooth functions are continuous on $M$ by definition).

Note that
\begin{equation}
\label{equation:h-equivalence-normes}
h^{4N} \|u\|^2_{H^{2N}_{1,\leaf}(M)} \leq \|u\|^2_{H^{2N}_{h,\leaf}(M)} \leq \|u\|^2_{H^{2N}_{1,\leaf}(M)}
\end{equation}
which implies that the completions of $C^\infty(M)$ with respect to these norms are independent on $h$. In the following, we will denote by $\boldsymbol{\Delta}_L$ the restriction of the leafwise Laplacian $\boldsymbol{\Delta}_{\mc{F}}$ to the leaf $L \in \mc{F}$.

For a given $x \in M$, we also introduce the following leafwise norms (the second one is only a semi-norm):
\[
\begin{split}
\|u\|^2_{H^{2N}_h(L(x))} & := \|(1+h^2\boldsymbol{\Delta}_{L(x)})^N u\|^2_{L^2(L(x))}, \\
\|u\|^2_{\dot{H}^{2N}_h(L(x))} & := \|h^{2N}\boldsymbol{\Delta}_{L(x)}^N u\|^2_{L^2(L(x))}, \\
\|u\|^2_{\overline{H}^{2N}_h(L(x))} & := \|u\|^2_{L^2(L(x))} + \|u\|^2_{\dot{H}^{2N}_h(L(x))}.
\end{split}
\]
A similar inequality as \eqref{equation:h-equivalence-normes} also holds for $\|\bullet\|_{H^{2N}_h(L(x))}$. If $U \subset L(x)$ is an open subset, we may also write $\|u\|^2_{H^{2N}_h(U)} :=  \|(1+h^2\boldsymbol{\Delta}_{L(x)})^N u\|^2_{L^2(U)}$, and the same definition holds for the other norms. We first prove that the norms defining the spaces $H^{2N}_h$ and $\overline{H}^{2N}_h$ are uniformly equivalent, with a constant that is independent of $h > 0$.

\begin{lemma}
\label{lemma:sobolev}
Let $N \geq 0$. There exists a constant $C_0 := C_0(N)$ such that for all $x \in M$, for all $u \in C^\infty_{\comp}(L(x))$:
\[
\|u\|^2_{\overline{H}^{2N}_h(L(x))} \leq \|u\|^2_{H^{2N}_h(L(x))} \leq C_0 \|u\|^2_{\overline{H}^{2N}_h(L(x))}.
\]
\end{lemma}

The proof is standard and therefore omitted. Similar norms can be defined for sections of a (leafwise) smooth vector bundle $E \to M$; it suffices to replace $\boldsymbol{\Delta}_L$ by $\boldsymbol{\Delta}^E_L$ (defined in \eqref{equation:laplacien}) in the paragraphs above. Finally, we will need the following:

\begin{lemma}\label{lemma-quickcomputation} 
Let $\chi \in C^\infty_{\comp}(L)$ be a smooth, nonnegative, compactly supported function on a leaf $L$. Let $(\chi_t)_{t\in [0,1]}$ be a continuous family of smooth, non-negative functions so that $\chi_t \leq \chi$ and so that $\chi_t$ is supported in $\{\chi=1\}$. Then, there exists $C > 0$ so that for every $f \in H^{2N}_h(L)$ and $t \in [0,1]$, one has
$$ \| \chi_t f \|_{H^{2N}_h(L)} \leq C \| \chi f \|_{H^{2N}_h(L)}. $$
\end{lemma}

\begin{proof}
Since $\chi_t \chi = \chi_t$, we may use Leibnitz rule to get 
\[
\begin{split}
\|\chi_t f\|_{H^{2N}} &
\leq C \sum_{|\alpha|\le 2N}
\|\partial^\alpha \chi_t\|_{L^\infty}
\|\chi f\|_{H^{2N}} \leq C_t \|\chi f\|_{H^{2N}},
\end{split}
\]
where $C_t := C \sum_{|\alpha|\le 2N} \|\partial^\alpha \chi_t\|_{L^\infty}$. The constant $C_t > 0$ varies continuously for $t \in [0,1]$ so there is a uniform constant as claimed.
\end{proof}

\subsubsection{Dual spaces} We can now introduce the dual spaces $H^{-2N}_h(L(x))$ for $N \geq 0$. Given an open subset $U \subset L(x)$, we define for $u \in C^\infty(L(x))$:
\[
\|u\|_{H^{-2N}_h(U)} := \sup_{\substack{\varphi \in C^\infty_{\comp}(U),\\\|\varphi\|_{H^{2N}_h(U)}=1}} (u, \varphi \mu_x).
\]
We then define $H^{-2N}_h(U)$ as the completion of $C^\infty(U)$ with respect to this norm. Note that, by construction, for $u \in H^{-2N}_h(U)$ and $v \in H^{2N}_{h,\comp}(U)$, one has:
\begin{equation}
\label{equation:fourche}
\langle u,v \rangle_{L^2(L)} = (u, \overline{v} \mu_x) \leq \|u\|_{H^{-2N}_h(U)} \|v\|_{H^{2N}_h(U)} \leq \dfrac{1}{2}\left( \|u\|_{H^{-2N}_h(U)}^2 + \|v\|_{H^{2N}_h(U)}^2 \right).
\end{equation}

\subsubsection{$L^2$-boundedness}

In what follows, we shall only need a (very) weak version of the Calderon-Vaillancourt theorem (see \cite[Theorem 13.13]{Zworski-12} for instance), asserting that semiclassical pseudodifferential operators of order $0$ are uniformly bounded, independently of $h > 0$:

\begin{lemma}
\label{lemma:uniform-bound}
Let $\mathbf{A} \in \Psi^0_{h,\leaf}(M,\mc{F})$. Then there exists $C > 0$ such that for all leaves $L \in \mc{F}$, $\|\mathbf{A}_L\|_{L^2(L) \to L^2(L)} \leq C$. In addition, if $\mc{F}$ is absolutely continuous (see Definition \ref{definition:foliation}), then $\|\mathbf{A}\|_{L^2(M) \to L^2(M)} \leq C$.
If $\mathbf{K}$ is residual, then $\|\mathbf{K}_L\|_{L^2(L) \to L^2(L)} =\mc{O}(h^\infty)$ uniformly in $L \in \mc{F}$, and $\|\mathbf{K}\|_{L^2(M) \to L^2(M)}= \mc{O}(h^\infty)$ when $\mc{F}$ is absolutely continuous.
\end{lemma}

We emphasize that a much stronger result actually holds, although we will not need it. Namely for $\mathbf{A} \in \Psi^0_{h,\leaf}(M,\mc{F})$, it is possible to establish that there exists $C > 0$ such that
\[
\|\mathbf{A}\|_{L^2(M) \to L^2(M)} \leq \|\sigma_{\mathbf{A}}\|_{L^\infty(T^*\mc{F})} + Ch,
\]
and for all $L \in \mc{F}$
\[
\|\mathbf{A}_L\|_{L^2(L) \to L^2(L)} \leq \|\sigma_{\mathbf{A}}|_L\|_{L^\infty(T^*\mc{F}|_L)} + Ch.
\]
More generally, operators of order $m \in \R$ are (uniformly) bounded between any Sobolev space $H^{s+m}_h(L) \to H^s_h(L)$ for $s \in \R$; we do not prove it as it will not be used in what follows. The proof of Lemma \ref{lemma:uniform-bound} uses the fact that pseudodifferential operators in $\Psi^m_{h,\leaf}(M,\mc{F})$ (in particular, residual operators) have uniformly bounded off-diagonal support.

\begin{proof}
We first prove the claim for residual operators. Let $\mathbf{K}$ be residual. Consider $M = \cup_{i} U_i$, a cover of $M$ by admissible charts (see \eqref{equation:admissible-coordinates}), and $\mathbf{1} = \sum_{i} \chi_i$ a partition of unity subordinated to this cover. We decompose $\mathbf{K} = \sum_{i} \chi_i \mathbf{K}$. It suffices to prove that each term in the sum is $\mc{O}(h^\infty)$ as an operator on $L^2(M)$. For that, in admissible coordinates, under the appropriate identifications, using the definition of residual operators (see Definition \ref{definition:residual}) and the fact that the foliation is absolutely continuous (see \eqref{equation:utile}), we have:
\[
\begin{split}
\|\chi_i \mathbf{K} u\|^2_{L^2(M)} & = \int_{V_1} \int_{V_2} |\chi_i \mathbf{K} u(x)|^2 J(x) \dd x_2 \dd x_1 \\
& =  \int_{V_1} \int_{V_2} \left|\chi_i \int_{L(x)} K(x,y) u(y) \dd \vol(y)\right|^2 J(x) \dd x_2 \dd x_1 \\
& \leq h^N \int_{V_1} \int_{V_2} \int_{y \in L(x) : d_{L(x)}(y,x) \leq C} |u(y)|^2 \dd \vol(y) J(x) \dd x_2 \dd x_1,
\end{split}
\]
where $N > 0$ is arbitrary, and we have used in the last line that the support of $K(x,\bullet)$ is contained in $\{y \in L(x) ~:~ d_{L(x)}(y,x) \leq C\}$ for some uniform constant $C > 0$ (in particular, it has finite and uniformly bounded volume), and that the $C^0$-norm of $K$ is $\mc{O}(h^\infty)$, uniformly on each leaf. In turn, writing $u = \sum_i \chi_i u$, decomposing $u$ on each patch of admissible coordinates, the last term in the previous sequence of inequalities is easily to be bounded by $\leq \|u\|^2_{L^2(M)}$. This proves the claim.

Now, consider an operator $\mathbf{A} = \Op_{h,\leaf}(a) + \mathbf{K} \in \Psi^0_{h,\leaf}(M,\mc{F})$. By the previous paragraph, it suffices to estimate the $L^2$-norm of the operator $\Op_{h,\leaf}(a)$ and we follow Hörmander's standard trick for that. We let $C_0 > \sup_{(x,\xi) \in T^*\mc{F}} |\sigma_{\mathbf{A}}(x,\xi)|$. For all $N \geq 0$, using the basic properties of the calculus, it is possible to prove (by induction on $N$) that
\[
C_0^2 = \Op_{h,\leaf}(a)^*\Op_{h,\leaf}(a) + \Op_{h,\leaf}(b_N)^*\Op_{h,\leaf}(b_N) + h^N \Op_{h,\leaf}(r_N),
\]
where $b_N \in S^0_{h,\leaf}(T^*\mc{F})$ and $r_N \in S^{-N}_{h,\leaf}(T^*\mc{F})$. We then have:
\[
\begin{split}
\|\Op_{h,\leaf}(a)u\|^2_{L^2(M)} & = \langle  \Op_{h,\leaf}(a)^*\Op_{h,\leaf}(a)u,u\rangle_{L^2(M)} \\
& = C_0^2 \|u\|^2_{L^2(M)} - \|\Op_{h,\leaf}(b_N)u\|^2_{L^2(M)} - h^N \langle  \Op_{h,\leaf}(r_N)u,u\rangle_{L^2(M)} \\
& \leq C_0^2\|u\|^2_{L^2(M)} + h^N\|\Op_{h,\leaf}(r_N)u\|_{L^2(M)}\|u\|_{L^2(M)}.
\end{split}
\]
It thus suffices to prove that for $N > 0$ large enough, $\|\Op_{h,\leaf}(r_N)\|_{L^2(M) \to L^2(M)} \leq C$ for some uniform constant $C > 0$. This can be done for $N > d$ by following the proof of the previous paragraph (based on Schur's lemma) and using the explicit expression \eqref{equation:leafwise-rn} for the kernel of the operator $\Op_{h,\leaf}(r_N)$ in each admissible coordinate patch (notice that for $N > d$, the Schwartz kernel is a continuous function on each leaf). The same arguments also apply to bound the operators restricted to each leaf $L \in \mc{F}$, but does not require absolute continuity of the foliation to pass from an $L^2$ estimate on each leaf to the global $L^2$ estimate on $M$. This proves the claim. 
\end{proof}

\subsubsection{A useful consequence of sharp Gårding}

In the following, we will apply the sharp Gårding inequality (Lemma \ref{lemma:garding0}) in a specific form involving two operators. We record it here for later use. Recall that $\mathbf{A} \in \Psi^m_{h,\mathrm{leaf}}(M,E)$ is principally diagonal if $\sigma_{\mathbf{A}} = a \cdot \mathbf{1}_E$ (see the end of \S\ref{sssection:main-prop}).

\begin{lemma}
\label{lemma:garding}
Let $\mathbf{A}, \mathbf{B} \in \Psi^{2m}_{h,\mathrm{leaf}}(M,E)$, with $m \in \Z_{\geq 0}$, be two principally diagonal operators such that $0 \leq \sigma_{\mathbf{A}} \leq \sigma_{\mathbf{B}}$. Then there exists $C > 0$ such that for all leaves $L \in \mc{F}$, for all $u \in C^\infty_{\comp}(L,E)$:
\[
\|\mathbf{A}_L u\|^2_{L^2(L,E)} \leq \|\mathbf{B}_L u\|^2_{L^2(L,E)} + C h \|u\|^2_{H^{2m}_{h}(L,E)}
\]
In addition, if $\mc{F}$ is absolutely continuous, for all $u \in C^\infty(M,E)$:
\[
\|\mathbf{A} u\|^2_{L^2(M,E)} \leq \|\mathbf{B} u\|^2_{L^2(M,E)} + C h \|u\|^2_{H^{2m}_{h,\mathrm{leaf}}(M,E)}.
\]
\end{lemma}

The lemma actually holds for all $m \in \R$ but do not provide a proof here. Additionally, the remainder term in both inequalities of Lemma \ref{lemma:garding} is suboptimal and could be improved to $\|u\|^2_{H^{2m-1/2}_{h,\mathrm{leaf}}(M,E)}$. We indicate in the proof where the loss occurs; however, for applications, we do not need this gain of $1/2$ of Sobolev regularity. Once again, we abuse notation here by identifying $\sigma_{\mathbf{A}}$ with a scalar function.

\begin{proof}
The principal symbol of $\mathbf{C} := \mathbf{B}^*\mathbf{B} - \mathbf{A}^*\mathbf{A} \in \Psi^{4m}_{h,\leaf}(M,E)$ satisfies
\begin{equation}
\label{equation:crux}
\sigma_{\mathbf{C}}(x,\xi) = \sigma_{\mathbf{B}}^2(x,\xi) - \sigma_{\mathbf{A}}^2(x,\xi) \geq 0.
\end{equation}
Applying Lemma \ref{lemma:garding0}, we obtain the existence of $\mathbf{K} \in \Psi^{4m-1}_{h,\leaf}(M,E)$ such that $\mathbf{C} + h \mathbf{K} \geq 0$ in the sense of quadratic form on $L^2(M,E)$ and $\mathbf{C}_L + h \mathbf{K}_L \geq 0$ on $L^2(L,E)$ for all leaves $L \in \mc{F}$. Observing that $\langle \mathbf{C} u,u\rangle_{L^2(M,E)} = \|\mathbf{B}u\|^2_{L^2(M,E)} - \|\mathbf{A}u\|^2_{L^2(M,E)}$, we then find that
\[
\|\mathbf{A}_L u\|^2_{L^2(L,E)} \leq \|\mathbf{B}_L u\|^2_{L^2(L,E)} + h \langle \mathbf{K}_L u, u\rangle_{L^2(L,E)},
\]
and
\[
\|\mathbf{A} u\|^2_{L^2(M,E)} \leq \|\mathbf{B} u\|^2_{L^2(M,E)} +  h \langle \mathbf{K} u, u\rangle_{L^2(M,E)}.
\]
It remains to bound:
\begin{equation}
\label{equation:marcelle}
\langle \mathbf{K} u, u\rangle_{L^2(L,E)} \leq C\|u\|^2_{H^{2m}_h(L,E)}, \quad \langle \mathbf{K} u, u\rangle_{L^2(M,E)} \leq C\|u\|^2_{H^{2m}_{h,\mathrm{leaf}}(M,E)},
\end{equation}
for some uniform constant $C > 0$. 

Let us consider a parametrix $\mathbf{B} \in \Psi^{-2m}_{h,\leaf}(M,\mc{F}, E \to E)$ for the operator $(1+h^2\boldsymbol{\Delta})^m$ (see Lemma \ref{lemma:parametrix}) such that
\[
\mathbf{B} (1+h^2\boldsymbol{\Delta})^m = 1 + \mathbf{S},
\]
where $\mathbf{S}$ is residual. Considering the restriction to the leaf $L$, and taking the adjoint, we have:
\[
\mathbf{B}_L (1+h^2\boldsymbol{\Delta}_L)^m = 1 + \mathbf{S}_L, \qquad  (1+h^2\boldsymbol{\Delta}_L)^m (\mathbf{B}_L)^* = 1 + (\mathbf{S}_L)^*,
\]
where both operators $\mathbf{S}_L$ and $(\mathbf{S}_L)^*$ are $\mc{O}(h^\infty)$ and uniformly smoothing.

We then write: 
\[
\begin{split}
\mathbf{K}_{L} = \left((1+h^2\boldsymbol{\Delta}_L)^m (\mathbf{B}_L)^* - (\mathbf{S}_L)^*\right)\mathbf{K}_{L}\left(\mathbf{B}_L (1+h^2\boldsymbol{\Delta}_L)^m -\mathbf{S}_L\right),
\end{split}
\]
which yields
\begin{equation}
\label{equation:sioux0}
\begin{split}
&\langle \mathbf{K}_{L} u, u\rangle_{L^2(L,E)}  = \langle (\mathbf{B}_L)^*  \mathbf{K}_{L} \mathbf{B}_L (1+h^2\boldsymbol{\Delta}_L)^m u, (1+h^2\boldsymbol{\Delta}_L)^m u\rangle_{L^2(L,E)} + \langle \mathbf{T}_L u, u \rangle_{L^2(L,E)},
\end{split}
\end{equation}
where $\mathbf{T}$ is $\mc{O}(h^\infty)$ and uniformly smoothing (and its expression involves all the above operators). 

Using Lemma \ref{lemma:uniform-bound} and \eqref{equation:h-equivalence-normes}, we have that 
\[
|\langle \mathbf{T}_L u, u \rangle_{L^2(L,E)}| = \mc{O}(h^\infty)\|u\|^2_{L^2(L,E)} \leq \mc{O}(h^\infty)\|u\|^2_{H^{2m}_{h,\leaf}(L,E)}.
\]
As to the first term in \eqref{equation:sioux0}, observe that, by Cauchy-Schwarz:
\[
\begin{split}
 |\langle (\mathbf{B}_L)^* \mathbf{K}_{L} \mathbf{B}_L & (1+h^2\boldsymbol{\Delta}_L)^m u, (1+h^2\boldsymbol{\Delta}_L)^m u\rangle_{L^2(L,E)}| \\
& \leq \|(\mathbf{B}_L)^*\mathbf{K}_{L}\mathbf{B}_L (1+h^2\boldsymbol{\Delta}_L)^m u\|_{L^2(L)} \|(1+h^2\boldsymbol{\Delta}_L)^m u\|_{L^2(L,E)} \\
& \leq \|(\mathbf{B}_L)^* \mathbf{K}_{L} \mathbf{B}_L \|_{L^2(L) \to L^2(L)} \|(1+h^2\boldsymbol{\Delta}_L)^m u\|^2_{L^2(L,E)} \\
&\leq C \| u\|^2_{H^{2m}_{h,\leaf}(L,E)},
\end{split}
\]
where we have used in the last line the definition of the $H^{2m}_{h,\leaf}(L,E)$ norm and that the operator 
\[
(\mathbf{B}_L)^* \mathbf{K}_{L} \mathbf{B}_L \in \Psi^{-1}_{h,\leaf}(M,\mc{F}, E \to E)
\]
has uniform bound\footnote{Note that we lose here some Sobolev exponent (we use the boundedness on $L^2(L)$ of an operator of order $-1$, which is suboptimal).} $\leq C$ on $L^2(L)$ by Lemma \ref{lemma:uniform-bound}. This proves the first inequality in \eqref{equation:marcelle}, and the second follows by a similar argument.
\end{proof}

\section{Analytic, dynamical and geometric preliminaries} \label{section:preliminaries}

Throughout this section, $M$ is a smooth closed (compact, without boundary) manifold endowed with an arbitrary background metric $g$. This section is organized as follows:
\begin{itemize}
\item In \S\ref{ssection:dynamics}, we recall standard facts on flows and discuss the propagator of a linear lift of the flow to a leafwise smooth vector bundle $E \to M$;  
\item In \S\ref{ssection:distributions}, we discuss the theory of distributions in the context of a foliated space;
\item In \S\ref{ssection:resonances}, we recall the definition of Pollicott-Ruelle resonances and resonant states along with some of their main properties.
\end{itemize}

\subsection{Dynamics} \label{ssection:dynamics} Let $\varphi_t : M \to M$ be a smooth flow with generator $X := \partial_t \varphi_t|_{t=0} \in C^\infty(M,TM)$. Let $E \to M$ be a leafwise smooth vector bundle and $\boldsymbol{\varphi}_t : E \to E$ be a linear cocycle over $(\varphi_t)_{t \in \R}$.

\subsubsection{Expanded foliation} \label{sssection:expanded-foliation} We say that the foliation $\mc{F}$ is \emph{expanded} by $\varphi_t$ if there exists $C, \lambda > 0$ such that
\begin{equation}
\label{equation:lambda2-introduction}
|\dd\varphi_{-t}(x)v| \leq C e^{-\lambda t} |v|, \qquad \forall v \in T\mc{F}(x), t \geq 0.
\end{equation}
From now on, we use the notation $W^u(x)$ in place of $L(x)$ to denote the leaves of $\mc{F}$, in order to be consistent with the notation in (partially) hyperbolic dynamical systems. We also introduce $E_u(x) := T_x\mc{F}$.

\begin{example}[Partially hyperbolic flows] Partially hyperbolic flows are typical examples of flows admitting an expanded foliation. For such flows, there exists a flow-invariant continuous splitting
\begin{equation}
\label{equation:splitting-tm}
TM = E_0 \oplus E_s \oplus E_u
\end{equation}
and constants $C,\lambda > 0$ satisfying that for every $x\in M$ and $t\geq 0$, if $v_s \in E_s(x), v_u \in E_u(x)$ and $v_0 \in E_0(x)$ are unit vectors then:
\begin{equation}
\label{equation:anosov}
\begin{aligned}
\lvert \dd\varphi_t(x)\, v_s \rvert
&\leq C e^{-\lambda t}, \\
\lvert \dd\varphi_{-t}(x)\, v_u \rvert
&\leq C e^{-\lambda t}, \\
\max \left\{
\frac{\lvert \dd\varphi_t(x)\, v_0 \rvert}{\lvert \dd\varphi_t(x)\, v_u \rvert},
\frac{\lvert \dd\varphi_{-t}(x)\, v_0 \rvert}{\lvert \dd\varphi_{-t}(x)\, v_s \rvert}
\right\}
&\leq C e^{-\lambda t}.
\end{aligned}
\end{equation}

As an example, Anosov flows satisfy \eqref{equation:splitting-tm} with $E_0$ reduced to $\R X$, the span of the flow generator. It is a non-trivial fact (see \cite{HirschPughShub}) that the strong unstable (or stable) foliation $\mc{F} = W^u$ of an Anosov flow is a transversally continuous foliation with smooth leaves (Definition \ref{definition:foliation}). In addition, the foliation is absolutely continuous. We refer to \cite[\S 2.2]{burns-wilkinson} and references therein for details.
\end{example}

Let $g$ be a smooth arbitrary background metric on $M$. Given $x \in M$, the leaf $W^u(x)$ is a smooth immersed submanifold of $M$. The restriction of $g$ to $W^u(x)$ defines a smooth metric $g|_{W^u(x)}$ with Levi-Civita connection $\nabla_{W^u}$. Notice that this metric is uniformly bounded, namely $\nabla_{W^u}^k \mc{R}|_{TW^u}$ (the Riemann curvature tensor) is uniformly bounded (in the $C^0$ topology) on the whole leaf for any $k\geq 0$.

\subsubsection{Propagator} Recall that the \emph{propagator} of $\X$ is the group of operators
\[
e^{-t\X} : C^\infty_{\leaf}(M,E) \to C^\infty_{\leaf}(M,E)
\]
such that
\[
e^{-t\X} f (x) = \boldsymbol{\varphi}_t(f(\varphi_{-t}x)) \in E_x,  \qquad \forall x \in M, f \in C^\infty_{\leaf}(M,E).
\]
Notice that $e^{-t\X}$ (for $t \geq 0$) corresponds to the \emph{forward propagation} of sections while $e^{t\X}$ corresponds to backward propagation.

The propagator $e^{-t\X} : E \to E$ is exponentially bounded on $L^\infty(M,E)$, that is, there exist $C,\lambda > 0$ such that for all $t \geq 0$:
\begin{equation}
\label{equation:e-infini0}
\|e^{-t \X}\|^2_{L^\infty(M,E)} \leq C e^{\lambda t}.
\end{equation}
This bound can be easily obtained by observing that $e^{-t\X}$ is uniformly bounded on $L^\infty(M,E)$ for all $t \in [0,1]$ and that $e^{-t\X}$ satisfies the group property $e^{-(t+s)\X} = e^{-t\X} e^{-s\X}$. Equivalently, \eqref{equation:e-infini0} can be stated as the following fiberwise estimate (here, $|\bullet|$ denotes the norm on $E$):
\begin{equation}
\label{equation:e-infini}
|\boldsymbol{\varphi}_t(f_x)|^2 \leq C e^{\lambda t} |f_x|^2, \qquad \forall x \in M, f_x \in E_x, t \geq 0.
\end{equation}

Finally, we record for later use a straightforward consequence of Egorov's lemma in the leafwise calculus (see Lemma \ref{lemma:egorov-leafwise}):

\begin{lemma}
\label{lemma:egorov}
Let $\mathbf{A} \in \Psi^m_{h,\leaf}(M,\mc{F},E \to E)$ be a principally diagonal operator. Then
\[
e^{t\X} \mathbf{A} e^{-t\X} \in \Psi^m_{h,\leaf}(M,\mc{F},E \to E)
\]
is principally diagonal with principal symbol $\sigma_{e^{t\X} \mathbf{A} e^{-t\X}}(x,\xi) = \sigma_{\mathbf{A}}(\varphi_t(x), d\varphi_{t}^{-\top}(x)\xi)$.
\end{lemma}

Here, we slightly abuse notation by identifying as usual the endomorphism-valued symbol with a scalar.

\begin{proof}
Immediate consequence of Lemma \ref{lemma:egorov-leafwise} in the vector-valued case.
\end{proof}

\subsubsection{Unstable Jacobian} \label{sssubection:unstable-jacobian}
Given a vector bundle $V \to M$, we denote by $\Omega^1 V := |\Lambda^{\mathrm{rk}(V)} V^*|$ the density bundle of $V$ over $M$, that is the space of antisymmetric maps $\omega : V \times ... \times V \to \R$ such that
\[
\omega(A v_1,..., A v_{\mathrm{rk}(V)}) = |\det A| \omega(v_1, ...,  v_{\mathrm{rk}(V)}),
\]
for all $v_1, ... , v_{\mathrm{rk}(V)}\in V$ and $A$, a linear transformation of $V$. For $x \in M$, we let
\[
\vol_u(x) \in \Omega^1 E_u(x)
\]
be the Riemannian density on $E_u(x)$ induced by $g|_{W^u(x)}$, that is on the tangent space to the leaf $W^u(x)$.

Given $x \in M$ and $t \geq 0$, the map
\[
\varphi_t : W^u(\varphi_{-t} x) \to W^u(x)
\]
is a (smooth) diffeomorphism. We let $|\det_u \dd \varphi_t|$ be the Jacobian of this map—called the unstable Jacobian— such that
\[
\varphi_t^*(\vol_u(\varphi_t(y))) = |{\det}_u \dd\varphi_t(y)| \vol_u(y), \qquad \forall y \in W^u(\varphi_{-t}x).
\]
Observe that there exist $C,\lambda' > 0$ such that for all $x \in M, t \geq 0$:
\begin{equation}
\label{equation:bounded-jacobian}
\|{\det}_u \dd \varphi_t\|_{L^\infty(M)} \leq C e^{\lambda t}.
\end{equation}
This follows from \eqref{equation:e-infini} applied with $E = \Omega^1 E_u$.

\subsubsection{Propagator estimate on $L^2$} 

Given $x \in M$, the norm $\|\bullet\|_{L^2(W^u(x),E)}$ is defined as follows:
\[
\|f\|_{L^2(W^u(x),E)}^2 := \int_{W^u(x)} |f(y)|^2 \dd \vol_u(y),
\]
where $|\bullet|$ is the fiberwise norm on $E$. We will need the following estimate on $L^2$:

\begin{lemma}
There exist $C_1,\lambda_1 > 0$ such that for all $x \in M$, $t \geq 0$, $f \in L^2(W^u(\varphi_{-t}x), E)$, the following estimate holds:
\begin{equation}
\label{equation:growth1}
\|e^{-t\X} (f|_{W^u(\varphi_{-t}x)})\|^2_{L^2(W^u(x),E)} \leq C_1 e^{\lambda_1 t} \|f\|^2_{L^2(W^u(\varphi_{-t}x),E)}.
\end{equation}
Here, $e^{-t\X}$ is seen as a map from $L^2(W^u(\varphi_{-t}x), E)$ to $L^2(W^u(x))$.
\end{lemma}

\begin{proof}
Recall that the map $\varphi_t : W^u(\varphi_{-t}x) \to W^u(x)$ is a diffeomorphism with bounded (unstable) Jacobian, see \eqref{equation:bounded-jacobian}. Given $f \in C^\infty_{\comp}(W^u(\varphi_{-t}x),E)$, one has:
\[
\begin{split}
\|e^{-t\mathbf{X}}f\|_{L^2(W^u(x), E)}^2 &= \int_{W^u(x)} |\boldsymbol{\varphi}_t(f(\varphi_{-t}y))|^2 \dd \vol_u(y) \\
& \overset{\eqref{equation:e-infini}}{\leq} C e^{\lambda t} \int_{W^u(x)} |f(\varphi_{-t}y)|^2 \dd \vol_u(y) \\
& = C e^{\lambda t} \int_{W^u(\varphi_{-t}x)} |f(z)|^2 |{\det}_u \dd \varphi_t(z)|\dd \vol_u(z) \\
& \overset{\eqref{equation:bounded-jacobian}}{\leq} C_1 e^{\lambda_1 t}\int_{W^u(\varphi_{-t}x)} |f(z)|^2\dd \vol_u(z) = C_1 e^{\lambda_1 t} \|f\|_{L^2(W^u(\varphi_{-t}x),E)}^2,
\end{split}
\]
for some constants $C_1,\lambda_1 > 0$ (which can be made explicit from \eqref{equation:bounded-jacobian} and \eqref{equation:e-infini}). Here, we have applied the change of variable formula in the third line. By density of the space $C^\infty_{\comp}(W^u(\varphi_{-t}x),E)$ in $L^2(W^u(\varphi_{-t}x),E)$, the estimate remains valid on $L^2(W^u(\varphi_{-t}x), E)$, therefore establishing the claim.
\end{proof}

\subsubsection{Dual bundle} The natural action (or \emph{lift}) of the flow $\varphi_t : M \to M$ on $T^*M$ is by the inverse transpose of the differential $\dd \varphi_t^{-\top}$, where for all $x\in M, \xi \in T^*_xM$ and $v \in T_{\varphi_tx}M$,
\begin{equation}
\label{equation:transpose}
(\dd \varphi_t^{-\top}(x) \xi, v) := (\xi, \dd \varphi_t(x)^{-1} v).
\end{equation}
In the following, we denote the dual bundle $T^*\mc{F}$ to $T\mc{F}$ by $E_u^*$. Notice that $E_u^*$ is flow-invariant with respect to the action \eqref{equation:transpose} of the flow by inverse transpose. In addition, there exist $C_2,\lambda_2 > 0$ such that:
\begin{equation}
\label{equation:growth2}
|\dd \varphi_t^{-\top}(x)\xi| \leq C_2 e^{-\lambda_2 t} |\xi|, \qquad \forall \xi \in E_u^*(x), t \geq 0.
\end{equation}
Here $|\bullet|$ denotes the fiberwise metric on $T^*M$ induced by the metric $g$ on $TM$. The inequality \eqref{equation:growth2} follows immediately from \eqref{equation:lambda2-introduction} and $\lambda_2$ equals the exponent $\lambda$ in \eqref{equation:lambda2-introduction}.

\begin{example}[Dual splitting in the Anosov case] In the Anosov case, the splitting \eqref{equation:splitting-tm} admits a dual splitting on $T^*M$. Define the bundles $E_0^\perp, E_s^\perp, E_u^\perp \subset T^*M$ through the relations:
\[
E_0^\perp(E_s \oplus E_u) = 0 = E_s^\perp(E_0 \oplus E_s) = E_u^\perp(E_0 \oplus E_u).
\]
Similarly to \eqref{equation:splitting-tm}, one has a contraction (resp. expansion) property of the flow on $E_s^\perp$ (resp. $E_u^\perp$). In the literature, these bundles are usually denote by $E_0^*, E_s^*,E_u^*$; however, to avoid confusion with the dual of $E_0,E_s$ and $E_u$ (which are \emph{not} the same bundles, see the next remark), we opted for a new notation.
\end{example}

We end this paragraph with an important remark:

\begin{remark}[Identification of $E_s^\perp$ with $E_u^*$ in the Anosov case]
\label{remark:non}
Given $x \in M$ and $\xi : E_u(x) \to \C$, a $1$-form, $\xi$ can be extended to a $1$-form $\tilde{\xi} : T_xM \to \C$ by setting $\xi(v) = 0$ for all $v \in E_0 \oplus E_s$, that is $\tilde{\xi} \in E_s^\perp(x)$. Conversely, any $1$-form $T_xM \to \C$ vanishing on $E_0 \oplus E_s$ can be seen as a $1$-form on $E_u(x)$ by restriction. In other words, there is a natural identification of $E_s^\perp$ with the dual of $E_u$. Nevertheless, we emphasize that this map is only a \emph{topological isomorphism} of vector bundles, and does not preserve the (leafwise) smooth properties of these bundles. Indeed, $E_u$ is leafwise smooth along the strong unstable foliation $W^u$, and so is $E_u^*$. However, $E_s^\perp$ is \emph{not} leafwise smooth along $W^u$, because it is defined by the equation $E_s^\perp(E_0 \oplus E_s) = 0$, and $E_s$ is (in general) \emph{not} smooth along $W^u$.
\end{remark}

\subsection{Distributions and foliations} 

\label{ssection:distributions}

In this section, we discuss the restriction of distributions to a leaf, and a Fubini-type formula. Throughout \S\ref{ssection:distributions}, $\mc{F}$ is a transversally continuous foliation with smooth leaves of dimension $d \geq 1$. We let $d^\perp := n-d$ be the codimension. That $\mc{F}$ is expanded by a flow is not needed.

\subsubsection{Restriction of distributions to a leaf} We first establish that distributions satisfying a transverse wavefront set condition can be restricted to the leaves of the foliation. Recall that $d$ is the dimension of the leaves and $d^\perp = n-d$ is the codimension. In the following statement, $\chi_x \in C^\infty_{\comp}(L(x))$ is a small bump function around $x$, as introduced in \S\ref{sssection:leafwise-volume}.

\begin{lemma}
\label{lemma:restriction}
Let $N\geq 0, \eps > 0$ and $\Gamma \subset T^*M \setminus \{0\}$ be a closed cone such that $\Gamma \cap N^*\mc{F} = \emptyset$. Then there exists a uniform constant $C > 0$ such that for all $u \in H^{-2N+d^\perp/2+\eps}(M) \cap \mc{D}'_{\Gamma}(M)$, for all $x \in M$:
\[
\|\chi_x \mathbf{r}_{L(x)}(u)\|_{H^{-2N}(L(x))} \leq C \left(\|u\|_{H^{-2N+d^\perp/2+\eps}(M)} + \|u\|_{\Gamma,N,n}\right),
\]
where $\|\bullet\|_{\Gamma,N,n}$ is a certain semi-norm on $\mc{D}'_{\Gamma}(M)$ (which depends on $\Gamma,N$ and $n$).
\end{lemma}

Note that, under the hypothesis of the lemma, the restriction is always well-defined, see \cite[Lemma 4.3.2]{Lefeuvre-book}. That is $\mathbf{r}_{L(x)}$ extends to a continuous linear map on $\mc{D}'_{\Gamma}(M)$.

\begin{proof}
The proof boils down to a computation in local coordinates. Using a smooth diffeomorphism, the leaf $L(x)$ can be straightened to $\{x_1=0\} \subset \R^n$. In these coordinates, we set $u' := \chi_x u$. Then, for $\varphi \in C^\infty_{\comp}(L(x))$ with compact support around $x$, letting $\chi$ be a bump function depending only on the $x_2$-variable and such that $\chi(0)=1$, we find, using the Parseval identity:
\[
(\mathbf{r}_{L(x)} u', \varphi) = \lim_{\delta \to 0} \delta^{-d^\perp} (u', \varphi \otimes \chi(\bullet/\delta)) = \dfrac{1}{(2\pi)^n} \int_{\R^n}\widehat{u'}(\xi) \widehat{\varphi}(-\xi_2) \dd \xi_1 \dd \xi_2
\]
Inserting $(1+\xi_2^2)^{(d^\perp+\eps)/2}(1+\xi_1^2)^{-N}$ and applying the Cauchy-Schwarz inequality, we obtain:
\[
\begin{split}
|(\mathbf{r}_{L(x)} u', \varphi)| & \leq  \left(\int_{\R^{d}} |\widehat{\varphi}(-\xi_2)|^2(1+\xi_2^2)^{2N} \dd \xi_2 \int_{\R^{d_\perp}} (1+\xi_1^2)^{-(d^\perp/2+\eps)} \dd \xi_1\right)^{1/2} \\
& \qquad \times \left(\int_{\R^n}\dfrac{|\widehat{u'}(\xi)|^2(1+\xi_1^2)^{d^\perp/2+\eps}}{(1+\xi_2^2)^{2N}} \dd \xi_1 \dd \xi_2\right)^{1/2} \\
& \leq C \|\varphi\|_{H^{2N}(L(x))} \left(\int_{\R^n}\dfrac{|\widehat{u'}(\xi)|^2(1+\xi_1^2)^{d^\perp/2+\eps}}{(1+\xi_2^2)^{2N}} \dd \xi_1 \dd \xi_2\right)^{1/2},
\end{split}
\]
here $\eps > 0$ makes the integral convergent in the first line. To deal with the last integral, consider the cone $C := \{|\xi_1| \leq \eta |\xi_2|\}\subset \R^n$ for $\eta > 0$. Notice that $N^*\mc{F} = \{(x_1=0,x_2,\xi_1,0) ~|~ x_2 \in \R^{d}, \xi_1 \in \R^{d_\perp}\}$. Since $u \in \mc{D}'_{\Gamma}(M)$ and $\Gamma \cap N^*\mc{F} = \emptyset$, we obtain that $\Gamma \cap (\R^n \setminus C) = \emptyset$ if $\eta > 0$ is chosen large enough. This yields:
\[
\begin{split}
\int_{\R^n}\dfrac{|\widehat{u'}(\xi)|^2(1+\xi_1^2)^{d^\perp/2+\eps}}{(1+\xi_2^2)^{2N}}& \dd \xi_1 \dd \xi_2  = \int_C \bullet + \int_{\R^n \setminus C} \bullet \\
& \lesssim \int_{C}\dfrac{|\widehat{u'}(\xi)|^2}{(1+|\xi|^2)^{2N-d^\perp/2-\eps}} \dd \xi_1 \dd \xi_2 + \int_{\R^n \setminus C} \dfrac{\langle \xi \rangle^{-m} (1+\xi_1^2)^{d^\perp/2+\eps}}{(1+\xi_2^2)^{2N}} \\
& \lesssim \|u'\|_{H^{-2N+d^\perp/2+\eps}(M)} + \|u'\|_{m},
\end{split}
\]
where $m \gg 0$ is chosen large enough to make the integral converge, and we have used in the second integral that the Fourier transform of $u'$ decays faster than any polynomial power of $\langle\xi\rangle$ on $\R^n \setminus C$ under our hypothesis on $\WF(u')$. In the previous formula, $\|\bullet\|_m$ corresponds to a semi-norm on $\mc{D}'_{\Gamma}(M)$. We thus obtain:
\[
|(\mathbf{r}_{L(x)} u', \varphi)| \leq C \|\varphi\|_{H^{2N}(L(x))} (\|u'\|_{H^{-2N+d^\perp/2+\eps}(M)} + \|u'\|_{m}).
\]
This proves that
\[
\|\mathbf{r}_{L(x)} u'\|_{H^{-2N}(L(x))} \leq C(\|u'\|_{H^{-2N+d^\perp/2+\eps}(M)} + \|u'\|_{m}),
\]
which is the claimed result.
\end{proof}

Additionally, we will need a continuity result on $\chi_x \mathbf{r}_{L(x)}(u)$ with respect to $x \in M$. Recall from \S\ref{sssection:leafwise-volume} that $\mu_x$ is the Riemannian measure induced on the leaf $L(x) \in \mc{F}$. When $\mc{F}$ is transversally continuous with smooth leaves, the map $x \mapsto \chi_x \mu_x$ is continuous (Lemma \ref{lemma:leafwise-density}).

\begin{lemma}
\label{lemma:continuity-leafwise-integration}
Let $\Gamma \subset T^*M \setminus \{0\}$ be a closed conic subset such that $\Gamma \cap N^*\mc{F} = \emptyset$. Then for all $u \in \mc{D}'_\Gamma(M)$, $\varphi \in C^\infty(M)$, the function
\[
M \to \C, \qquad x \mapsto (\mathbf{r}_{L(x)}(u), \varphi \chi_x  \mu_x)
\]
is continuous.
\end{lemma}

\begin{proof}
Recall from Lemma \ref{lemma:leafwise-density} that $\chi_x \mu_x \in \mc{D}'_{N^*\mc{F}}(M,\Omega^1 M)$. Note that
\[
(\mathbf{r}_{L(x)}(u), \varphi  \chi_x \mu_x) = (u \times \chi_x \mu_x, \varphi),
\]
where $u \times \chi_x \mu_x$ is the product of distributions in $\mc{D}'_\Gamma(M) \times \mc{D}'_{N^*\mc{F}}(M,\Omega^1 M)$. This product is well-defined as $\Gamma \cap N^*\mc{F} = \emptyset$ (see \cite[Lemma 4.2.1]{Lefeuvre-book}). Additionally, as $x \mapsto \chi_x \mu_x \in \mc{D}'_{N^*\mc{F}}(M,\Omega^1M)$ is continuous (Lemma \ref{lemma:leafwise-density}), the continuity claimed in the lemma is immediate.
\end{proof}

\subsubsection{A Fubini formula for distributions} We now state a Fubini-type result for distributions. First, if $\mc{F}$ is absolutely continuous and $f \in C^\infty(M)$ is a smooth function with support in a small open subset $U \subset M$, we may write
\begin{equation}
\label{equation:intrinsic}
\int_U f(x) \dd \vol(x) = \int_{\Sigma} \left(\int_{L(x)} \chi_x(y) f(y) \dd \mu_x(y)\right) \dd \nu(x),
\end{equation}
where $\Sigma \subset M$ is a transverse slice to the foliation, the measure $\nu$ is a smooth measure on the slice $\Sigma$, and $\mu_x$ is a smooth measure on the leaf $L(x)$ such that $x \mapsto \mu_x$ is continuous (after identification of leaves via holonomy maps). Note that \eqref{equation:intrinsic} is a mere rewriting of the disintegration formula \eqref{equation:utile} in admissible coordinates, which holds thanks to the absolute continuity of the foliation. We also point out that the measure~$\mu_x$ in~\eqref{equation:intrinsic} does not necessarily coincide with the leafwise Riemannian measure introduced in \S\ref{sssection:leafwise-volume}. Nonetheless, the two differ only by multiplication with a leafwise smooth function, which will play no role in what follows. For this reason, we allow ourselves to use the same notation~$\mu_x$ for both. As in \S\ref{sssection:leafwise-volume}, we have that $M \ni x \mapsto \chi_x \mu_x \in \mc{D}'_{N^*\mc{F}}(M,\Omega^1M)$ is continuous (see Lemma \ref{lemma:leafwise-density}).

\begin{lemma}
Let $u \in \mc{D}'(M)$ such that $\WF(u) \cap N^*\mc{F} \setminus \{0\} = \emptyset$, and $\varphi \in C^\infty_{\comp}(U)$, where $U \subset M$ is the small open subset defined above. Then the function
\[
\Sigma \ni x \mapsto (\mathbf{r}_{L(x)}(u), \varphi \chi_x \mu_x) = \int_{L(x)} \mathbf{r}_{L(x)}(u)(y) \chi_x(y)\varphi(y) \dd \mu_x(y)
\]
is continuous on $\Sigma$. In addition, if $\mc{F}$ is absolutely continuous, then:
\begin{equation}
\label{equation:fubini}
(u,\varphi~\dd\vol) = \int_{\Sigma} (\mathbf{r}_{L(x)}(u), \varphi \chi_x \mu_x) \dd \nu(x).
\end{equation}
\end{lemma}

\begin{proof}
Continuity follows from Lemma \ref{lemma:continuity-leafwise-integration}. Let $\Gamma \subset T^*M \setminus \{0\}$ be an arbitrary closed cone such that $\WF(u) \subset \Gamma$ and $\Gamma \cap N^*\mc{F} = \emptyset$. Let $u_\eps \in C^\infty(M)$ be a sequence of smooth functions such that $u_\eps \to_{\eps \to 0} u$ in $\mc{D}'_\Gamma(M)$ (see \cite[Definition 4.1.8]{Lefeuvre-book} for a definition of the convergence in $\mc{D}'_\Gamma(M)$). It then holds that
\[
(u_\eps,\varphi~\dd\vol) \to_{\eps \to 0} (u,\varphi~\dd\vol), \qquad (\mathbf{r}_{L(x)}(u_\eps), \varphi  \chi_x \mu_x) \to_{\eps \to 0} (\mathbf{r}_{L(x)}(u), \varphi \chi_x  \mu_x).
\]
In addition, one verifies that $|(\mathbf{r}_{L(x)}(u_\eps), \chi_x \varphi \mu_x)| \leq C$ for all $\eps > 0$, where $C > 0$ is a uniform constant. Applying \eqref{equation:intrinsic} with $u_\eps$ and using the Lebesgue dominated convergence on the right-hand side, we find that \eqref{equation:fubini} holds for $u \in \mc{D}'(M)$ satisfying the assumptions of the lemma. 
\end{proof}

\subsection{Pollicott-Ruelle resonances}

 \label{ssection:resonances} 
In this paragraph, we further assume that the flow $(\varphi_t)_{t \in \R}$ is Anosov and $E \to M$ is globally smooth. The resolvents
\[
\mathbf{R}_\pm(\lambda) := (-\X-\lambda)^{-1} : C^\infty(M,E) \to L^2(M,E)
\]
are well defined for $\Re(\lambda) \gg 0$ using the converging integrals
\[
(-\X-\lambda)^{-1} = - \int_0^{+\infty} e^{\mp t\X} e^{-\lambda t} \dd t.
\]
It was established in \cite{Faure-Sjostrand-11, Dyatlov-Zworski-16} that these operators admits a meromorphic extension to $\C$ as maps
\[
\mathbf{R}_+(\lambda) : C^\infty(M,E) \to \mc{D}'_{E_u^\perp}(M,E), \qquad \mathbf{R}_-(\lambda) : C^\infty(M,E) \to \mc{D}'_{E_s^\perp}(M,E)
\]
with poles of finite rank (see also \cite[Chapter 9]{Lefeuvre-book}). These poles are called the \emph{Pollicott-Ruelle resonances}. To each resonances $\lambda_0 \in \C$, one can associate a spectral projector
\[
\Pi^\pm_{\lambda_0} := - \dfrac{1}{2i\pi} \int_\gamma \mathbf{R}_+(\lambda) \dd \lambda,
\]
where $\gamma$ is a small contour around $\lambda_0$. The range of $\Pi^\pm_{\lambda_0}$ is finite-dimensional, included in $\mc{D}'_{E_u^\perp}(M,E)$ (resp. $\mc{D}'_{E_s^\perp}(M,E)$) in the $+$ case (resp. in the $-$ case), and there exists an integer $\ell \geq 1$ such that
\[
\begin{split}
\Pi^+_{\lambda_0} & = \{u \in \mc{D}'_{E_u^\perp}(M,E) ~:~ (-\X-\lambda_0)^\ell u = 0\}, \\
 \Pi^-_{\lambda_0} & = \{u \in \mc{D}'_{E_s^\perp}(M,E) ~:~ (+\X-\lambda_0)^\ell u = 0\}.
 \end{split}
\]
Elements in the range of $\Pi_{\lambda_0}^+$ (resp. $\Pi_{\lambda_0}^-$) are called \emph{generalized Pollicott-Ruelle resonant states} (resp. \emph{coresonant states}).

\section{Proof of the regularity statements} \label{section:proofs}

The proofs of Theorems \ref{corollary:main} and \ref{theorem:main} make use of the leafwise semiclassical pseudodifferential calculus introduced in \S\ref{section:leafwise-calculus}. It is applied with $\mc{F} = W^u$ an expanding foliation for a flow $\varphi_t: M \to M$ on which there is a leafwise smooth bundle $E \to M$ and a fiberwise linear extension of the flow $\boldsymbol{\varphi}_t : E \to E$.

\subsection{Proof of propagation estimates}

\label{ssection:proof-propagation-estimates}

We begin our proof of Theorems \ref{corollary:main} and \ref{theorem:main} with a preliminary lemma. Given $x \in M$, we define
\[
\chi_x \in C^\infty_{\comp}(L(x)), \qquad \chi_x(y) := \chi(d_{L(x)}(x,y)),
\]
where $\chi \in C^\infty_{\comp}(\R)$ is a smooth non-negative bump function equal to $1$ on $[-1/2,1/2]$ and $0$ outside of $[-1,1]$, and $d_{L(x)}$ denotes the Riemannian distance induced on the leaf. In the following, $\boldsymbol{\Delta}_{W^u}$ denotes the leafwise Laplacian acting on sections of $E$  as defined in \eqref{equation:laplacien} (we drop the index $E$ for simplicity). When the operator is restricted to a single leaf $W^u(x)$, we shall use the notation $\boldsymbol{\Delta}_{W^u(x)}$. For $x \in M$  we also write $x(t) = \varphi_{-t}(x)$. 

The next lemma allows to compare the (local) Sobolev norms of $u$ on $W^u(\varphi_t x_\star)$ and $e^{-t \X} u$ on $W^u(x_\star)$. It is based on the following simple idea: as the flow expands (uniformly) the leaves $W^u$, its action on leafwise Sobolev spaces should be contracting.

\begin{lemma}
\label{lemma:key}
There exists $T_0 > 0$ such that for all $N \geq 0$ and $T \geq 0$, there exists $h_0, C_3 > 0$ such that for all $t \in [T_0,T]$, $h \in (0,h_0]$, $x_\star \in M$, $x \in W^u(x_\star)$, and for all $u \in \mc{D}'(W^u(x(t)),E)$ such that $\chi_{x(t)}u\in H^{2N}_{\mathrm{comp}.}(W^u(x(t)))$, the following inequality holds:
\[
\begin{split}
\| e^{t\X} h^{2N} &\boldsymbol{\Delta}^N_{W^u(x_\star)} e^{-t\X} [(\varphi_t^* \chi_x) u] \|_{L^2(W^u(x_\star(t)))}^2  \\
&\leq C_2^{4N} e^{-4 \lambda_2 N t} \|h^{2N} \boldsymbol{\Delta}^N_{W^u(x_\star(t))} [\chi_{x(t)} u]\|^2_{L^2(W^u(x_\star(t)))} + h C_3 \|\chi_{x(t)} u \|^2_{H^{2N}_h(W^u(x_\star(t)))}. 
\end{split}
 \] 
\end{lemma}

Recall that the constants $C_2, \lambda_2 > 0$ were defined in \eqref{equation:growth2}. Although it will not be needed for the purposes of what follows, we emphasize that the last term in the estimate of Lemma \ref{lemma:key} could be slightly improved to $h C_3 \|\chi_{x(t)} u \|^2_{H^{2N-1/2}_h(W^u(x_\star(t)))}$ (see the comment following Lemma \ref{lemma:garding}).

\begin{proof}
Observe that
\[
h^{2N} \boldsymbol{\Delta}^N_{W^u} = \Op_{h,\leaf}(|\xi|^{2N}_g) +\mc{O}_{\Psi^{2N-1}_{h,\mathrm{leaf}}}(h),
\]
where $|\bullet|_g^{2N} \in S^{2N}_{h,\leaf}(T^*W^u)$. From now on we will no longer include reference to the metric $g$. 

Notice that $h^{2N} \boldsymbol{\Delta}^N_{W^u}$ is principally diagonal. By the leafwise version of Egorov's theorem (Lemma \ref{lemma:egorov}), $e^{t\X} h^{2N} \boldsymbol{\Delta}^N_{W^u} e^{-t\X}$ is also principally diagonal and we have that:
\[
e^{t\X} h^{2N} \boldsymbol{\Delta}^N_{W^u} e^{-t\X} = \Op_h(|d\varphi_t^{-\top}(x(t))\xi|^{2N}) + \mc{O}_{\Psi^{2N-1}_{h,\mathrm{leaf}}}(h).
\]
Using \eqref{equation:growth2}, we see that
\[
|d\varphi_t^{-\top}(x(t))\xi|^{2N} \leq C_2^{2N} e^{-2 \lambda_2 N t} |\xi|^{2N}, \qquad \forall t \geq 0.
\]
Applying the consequence of the leafwise sharp Gårding inequality (see Lemma \ref{lemma:garding}) with the leaf
\[
L := W^u(x_\star(t)),
\]
and \eqref{equation:h-equivalence-normes} we find that for all $f \in H^{2N}_{\comp}(W^u( x_\star(t)))$, 
\[
\begin{split}
\|e^{t\X} h^{2N} \boldsymbol{\Delta}^N_{W^u(x_\star)} e^{-t\X} f\|_{L^2(L)}^2  \leq C_2^{4N} e^{-4\lambda_2 N t} \|h^{2N} \boldsymbol{\Delta}^N_{L} f\|^2_{L^2(L)} + h C(t) \|f\|^2_{H^{2N}_h(L)},
\end{split}
\] 
where the constant $C(t) > 0$ depends on time (however, for $t \in [0,T]$, $C(t)$ is uniformly bounded by a constant depending only on $T > 0$).

We then apply this inequality with $(\varphi_t^* \chi_x) u \in H^{2N}_{\comp}(L)$. We find that
\[
\begin{split}
\|e^{t\X} h^{2N} \boldsymbol{\Delta}^N_{W^u(x_\star)} e^{-t\X}(\varphi_t^* \chi_x) u\|_{L^2(L)}^2  & \leq C_2^{4N} e^{-4\lambda_2 N t} \|h^{2N} \boldsymbol{\Delta}^N_{L} (\varphi_t^* \chi_x) u\|^2_{L^2(L)} \\
&\qquad + h C(t) \|(\varphi_t^* \chi_x) u\|^2_{H^{2N}_h(L)}.
\end{split}
\]
By expansivity of the flow, for $t \geq T_0$ where $T_0$ is chosen large enough, one has $\varphi_t^* \chi_x = \chi_x \circ \varphi_{t} \leq \chi_{x(t)}$ and $\chi_x \circ \varphi_{t}$ is supported in $\{\chi_{x(t)}=1\}$. Lemma \ref{lemma-quickcomputation} then gives:
\[
\|(\varphi_t^* \chi_x) u\|^2_{H^{2N}_h(L)} \leq C'(t) \|\chi_{x(t)} u\|^2_{H^{2N}_h(L)},
\]
for some (other) constant $C'(t) > 0$. The same remark as above applies to $C'(t)$; it is uniformly bounded for $t \in [0,T]$. Inserting this bound in the previous estimate, we obtain the inequality claimed in the lemma. Notice that the constant $C_3$ is given by $C_3 = \sup_{t \in [0,T]} C(t) \cdot \sup_{t \in [0,T]} C'(t)$.
\end{proof}

The following estimate allows to bring the bounds back to a fixed leaf and obtain the contraction on the Sobolev norm. In other words, it shows that the induced action of the flow on Sobolev spaces (between two different leaves) is a contraction up to lower order/low frequency contributions (hence the presence of an $L^2$ norm on the right-hand side of the following inequality):

\begin{lemma}
\label{lemma:marcelle}
Let $N \geq 0$ be an integer. For all $T > 0$, there exists $C_3 := C_3(T) > 0$ such that for all $h > 0$ small enough, for all vector-valued distributions $u \in \mc{D}'(W^u( x(T)),E)$, for all $x \in M$ such that
\[
\chi_{ x(T)} u \in \overline{H}_h^{2N}(W^u(x(T)), E),
\]
one has
\[
\chi_x e^{-T\X} u \in \overline{H}_h^{2N}(W^u(x),E).
\]
In addition, the following inequality holds:
\[
\begin{split}
\| \chi_x e^{-T\X} u \|_{\overline{H}_h^{2N}(W^u(x))}^2 
& \leq \| \chi_x e^{-T\X} u \|_{L^2(W^u(x))}^2  +  \omega(T,h)  \|\chi_{ x(T)} u \|_{\overline{H}_h^{2N}(W^u(x(T)))}^2,
\end{split}
\]
where
\begin{equation}
\label{equation:omegath}
\omega(T,h) := C_1 C_2^{4N} e^{(\lambda_1-4 \lambda_2 N) T} + C_0C_1C_3 e^{\lambda_1 T} h.
\end{equation}
\end{lemma}

The constant $C_0$ is the one from Lemma \ref{lemma:sobolev}, $C_1$ is defined in \eqref{equation:growth1} and $C_2$ in \eqref{equation:growth2}.

\begin{proof}
We have:
\[ 
\begin{split}
\|\chi_x e^{-t\X} u\|^2_{\dot{H}^{2N}_h(W^u(x))} & = \| h^{2N} \boldsymbol{\Delta}^N_{W^u(x)} \chi_x e^{-t\X} u \|_{L^2(W^u(x))}^2 \\
& = \|  e^{-t\X} e^{t\X} h^{2N} \boldsymbol{\Delta}^N_{W^u(x)} e^{-t\X} [(\varphi_t^* \chi_x) u] \|_{L^2(W^u(x))}^2  \\
& \leq C_1 e^{\lambda_1 t} \| e^{t\X} h^{2N} \boldsymbol{\Delta}^N_{W^u(x)} e^{-t\X} [(\varphi_t^* \chi_x) u] \|_{L^2(W^u(x(t)))}^2,
\end{split}
\]
using \eqref{equation:growth1} in the last line.

Applying Lemma \ref{lemma:key} with this fixed time $T > 0$, we find that there exists $C_3 := C_3(T) > 0$ such that:
\[
\begin{split}
\|\chi_x e^{-T\X} u\|^2_{\dot{H}^{2N}_h(W^u(x))} & \leq C_1 C_2^{4N} e^{(\lambda_1-4 \lambda_2 N) T} \|h^{2N} \boldsymbol{\Delta}^N_{W^u(x_\star(T))} \chi_{x(T)} u\|^2_{L^2(W^u(x(T)))} \\
& \hspace{3cm}+ C_1 C_3(T) e^{\lambda_1 T}  h\| \chi_{x(T)}u \|^2_{H^{2N}_h(W^u(x(T)))} \\
& = C_1 C_2^{4N} e^{(\lambda_1-4 \lambda_2 N) T} \| u\chi_{ x(T)} \|_{\dot{H}_h^{2N}(W^u(x(T)))}^2 \\
& \hspace{3cm}+ C_1 C_3(T) e^{\lambda_1 T}  h \|\chi_{x(T)} u\|^2_{H^{2N}_h(W^u(x(T)))}.
\end{split}
 \]
 By Lemma \ref{lemma:sobolev}, there exists a constant $C_0 > 0$ such that 
 \[
\|\chi_{x(T)} u\|^2_{H^{2N}_h(W^u(x(T))} \leq C_0 \|\chi_{x(T)} u\|^2_{\overline{H}^{2N}_h(W^u(x(T))}
 \]
 Inserting this in the previous estimate, we find that:
 \begin{equation}
 \label{equation:cool0}
 \begin{split}
 \| \chi_x e^{-T\X} u \|_{\dot{H}_h^{2N}(W^u(x))}^2  \leq \underbrace{\left(C_1 C_2^{4N} e^{(\lambda_1-4 \lambda_2 N) T} + C_0C_1C_3 e^{\lambda_1 T} h \right)}_{=\omega(T,h)}  \|\chi_{ x(T)} u\|_{\overline{H}_h^{2N}(W^u( x(T)))}^2.
 \end{split}
 \end{equation}
Using the definition of the norm on $\overline{H}_h^{2N}$ (see \S\ref{ssection:sobolev-norms}), this proves the claimed estimate.
\end{proof}

We can now conclude the proof of the propagation estimates. Let us first consider the easier case of Theorem \ref{corollary:main}:

\begin{proof}[Proof of Theorem \ref{corollary:main}]
With $\eps = 1$, the inequality \eqref{equation:bound-main-corollary} boils down to
\begin{equation}
\label{equation:todo}
\| \chi_{x} u|_{W^u(x)}\|_{H^{2N}(W^u(x))} \leq C (\|u\|_{C^0(M,E)} + e^{-\nu t}\|\chi_{x_0} u|_{W^u(x_0)}\|_{H^{2N}(W^u(x_0))}).
\end{equation}
The proof of this estimate relies on an iteration of Lemma \ref{lemma:marcelle}. Define

\[
\mathbf{v}_1 := \sup_{x \in M} \vol_{g|_{W^u(x)}}(B_u(x,1)).
\]
Using that $e^{T\X} u = e^{\lambda T} u$ for some $\lambda \in \C$, Lemma \ref{lemma:marcelle} yields:
\begin{equation}
\begin{split}
\label{equation:nuit}
&\| \chi_x (u|_{W^u(x)}) \|_{\overline{H}_h^{2N}(W^u(x))}^2  \\
& \qquad   \leq \|\chi_x (u|_{W^u(x)})\|^2_{L^2(W^u(x))} + e^{2 \Re(\lambda) T} \omega(T,h) \|\chi_{ x(T)}( u|_{W^u(x(T))})\|_{\overline{H}_h^{2N}(W^u(x(T)))}^2 \\
& \qquad  \leq \mathbf{v}_1 \|u\|^2_{C^0(M,E)} + e^{2 \Re(\lambda) T}\omega(T,h)  \|\chi_{x(T)}( u|_{W^u(x(T))})\|_{\overline{H}_h^{2N}(W^u(x(T)))}^2.
 \end{split}
\end{equation}
We have:
\[
e^{2 \Re(\lambda) T}\omega(T,h) = C_1 C_2^{4N} e^{(2 \Re(\lambda) + \lambda_1-4 \lambda_2 N) T} + C_0C_1C_3 e^{(2 \Re(\lambda) + \lambda_1) T} h
\]
We fix an integer $N \geq 0$ such that $2 \Re(\lambda) + \lambda_1-4 \lambda_2 N < 0$, that is
\begin{equation}
\label{equation:n-threshold}
N > \dfrac{2\Re(\lambda) + \lambda_1}{4 \lambda_2} = \Re(\lambda)/2\lambda_2 + \boldsymbol{\mu}, \qquad \boldsymbol{\mu} := \lambda_1/4\lambda_2
\end{equation}
We then choose $T > 0$ such that 
\[
C_1 C_2^{4N} e^{(2 \Re(\lambda) + \lambda_1-4 \lambda_2 N) T} \leq 1/4.
\]
Finally, we\footnote{We note that it is the choice of $h > 0$ which makes our constants depend on higher order derivatives of the flow. Indeed, constants $C_1$ and $C_2$ can be chosen uniformly in a $C^1$-neighborhood of the flow. However, $C_3$ depends on higher order derivatives as it uses Lemma \ref{lemma:garding}.} choose $h > 0$ small enough such that 
\[
C_0C_1C_3 e^{(2 \Re(\lambda) + \lambda_1) T} h \leq 1/4.
\]
This implies that $e^{2 \Re(\lambda) T}\omega(T,h) \leq 1/2$ and thus by \eqref{equation:nuit}:

\begin{equation}
\label{equation:nuit2}
\| \chi_x( u|_{W^u(x)}) \|_{\overline{H}_h^{2N}(W^u(x))}^2 \leq \mathbf{v}_1 \|u\|^2_{C^0(M,E)} + \dfrac{1}{2} \|\chi_{ x(T)} (u|_{W^u(x(T))})\|_{\overline{H}_h^{2N}(W^u( x(T)))}^2.
\end{equation}

Now, suppose $x \in \varphi_{nT}(W^u(x_0,1/2))$. Define for $0 \leq k \leq n$:
\[
a_k := \|\chi_{x((n-k)T)} u|_{W^u(x((n-k)T))}\|^2_{\overline{H}^{2N}_h(W^u(x((n-k)T)))}.
\]
By \eqref{equation:nuit}, we have:
\[
a_{k+1} \leq \mathbf{v}_1 \|u\|_{C^0}^2 + a_k/2.
\]
This implies that $a_n \leq 2 \mathbf{v}_1 \|u\|_{C^0}^2 + a_0/2^n$, which is exactly \eqref{equation:todo} for times $t=nT$ with explicit constants $C$ and $\nu$. To pass to arbitrary times, it suffices to propagate for times between $0$ and $T$; this accounts for the non-explicit constant in the theorem. Observe, however, that this constant is uniform with respect to the flow. Finally, to prove \eqref{equation:bound-main-corollary} for $\eps \leq 1$, it suffices to use that the propagator $e^{t\X}$ has an exponential bound on $H^{2N}$; this accounts for the factor $\eps^{-\nu}$.

For the proof of \eqref{equation:bound-main2-corollary}, consider $x \in \Lambda_+(x_0,\eps)$ and a sequence $x_n \to x$, $x_n = \varphi_{t_n}(y_n)$ with $y_n \in W^u_{\eps/2}(x_0)$, $t_n \to +\infty$. Let $u_n := u|_{W^u_1(x_n)}$, and $u_\infty := u|_{W^u_1(x)}$. Applying \eqref{equation:bound-main-corollary}, we obtain:
\begin{equation}
\label{equation:train-rennes}
\|u_n\|_{H^{2N}(W^u_1(x_n))} \leq C\left(\|u\|_{C^0(M,E)} + \eps^{-\nu} e^{-\nu t_n} \|u\|_{H^{2N}(W^u_\eps(x_0))}\right).
\end{equation}
Let $\varphi \in C^\infty(M)$. Then $(u_n, \varphi \chi_{x_n} \mu_{x_n}) \to_{n \to +\infty} (u_\infty,\varphi \chi_x \mu_x)$ by continuity of $u \in C^0(M,E)$. In addition:
\[
\begin{split}
|(u_n,\varphi \chi_{x_n} \mu_{x_n})| & \leq \|u_n\|_{H^{2N}(W^u_1(x_n))} \|\varphi \chi_{x_n} \mu_{x_n}\|_{H^{-2N}(W^u_1(x_n))} \\
& \leq C'\|u_n\|_{H^{2N}(W^u_1(x_n))} \|\varphi\|_{H^{-2N}(W^u_1(x_n))},
\end{split}
\]
for some uniform constant $C' > 0$. Using \eqref{equation:train-rennes}, and passing to the limit in the previous inequality, we obtain
\[
|(u,\varphi \mu_x)| \leq C C' \|u\|_{C^0(M,E)} \|\varphi\|_{H^{-2N}(W^u_1(x))},
\]
which implies that $u \in H^{2N}(W^u_1(x))$ with norm $\leq CC'\|u\|_{C^0(M,E)}$. The constant $C$ in the statement of Theorem \ref{corollary:main} is the max between the constants $C$ and $CC'$ defined here. By construction, it is clear that they depend uniformly on the flow.
\end{proof}

The proof of Theorem \ref{theorem:main} is almost identical.

\begin{proof}[Proof of Theorem \ref{theorem:main}]
Let $u \in \mc{D}'_{\Gamma}(M,E)$ such that $(\X-\lambda) u = 0$ and assume that $u$ is $H^{2N}$ is restriction to $W^u_\eps(x_0)$. The same argument as in the proof of Theorem \ref{corollary:main} applies, and we find, similarly to \eqref{equation:nuit},
\begin{equation}
\begin{split}
\label{equation:rer}
&\| \chi_x u|_{W^u(x)} \|_{\overline{H}_h^{2N}(W^u(x))}^2  \\
& \qquad   \leq \|\chi_x u|_{W^u(x)}\|^2_{L^2(W^u(x))} + e^{2 \Re(\lambda) T} \omega(T,h) \|\chi_{\varphi_{-T} x} u|_{W^u(\varphi_{-T}x)}\|_{\overline{H}_h^{2N}(W^u(\varphi_{-T} x))}^2,
\end{split}
\end{equation}
where $\omega(T,h)$ was defined in \eqref{equation:omegath}. We then bound the $L^2$ norm on the right-hand side of \eqref{equation:rer} using \eqref{equation:fourche}:
\[
\|\chi_x u|_{W^u(x)}\|^2_{L^2(W^u(x))} \leq \dfrac{1}{2} \|\chi_x u\|^2_{\overline{H}^{-2N}(W^u(x))} + \dfrac{1}{2}\|\chi_x u\|^2_{\overline{H}^{2N}(W^u(x))}.
\]
Inserting the previous estimate in \eqref{equation:rer}, we obtain:
\begin{equation}
\label{equation:feu}
\begin{split}
&\| \chi_x u|_{W^u(x)} \|_{\overline{H}_h^{2N}(W^u(x))}^2  \\
& \qquad   \leq  \|\chi_x u\|^2_{\overline{H}^{-2N}(W^u(x))}  + 2 \cdot e^{2 \Re(\lambda) T} \omega(T,h) \|\chi_{\varphi_{-T} x} u|_{W^u(\varphi_{-T}x)}\|_{\overline{H}_h^{2N}(W^u(\varphi_{-T} x))}^2,
\end{split}
\end{equation}
As in \eqref{equation:n-threshold}, we assume that $N > \Re(\lambda)/2\lambda_2 + \boldsymbol{\mu}$. By Lemma \ref{lemma:restriction}, we have that
\[
 \|\chi_x u\|^2_{\overline{H}^{-2N}(W^u(x))} \leq C\left(\|u\|^2_{H^{-2N+d^\perp/2 + \gamma}} + \|u\|^2_{\Gamma,N,n}\right).
\]
Inserting the previous estimate in \eqref{equation:feu}, and taking $T \gg 1$ large enough to ensure $2 \cdot e^{2 \Re(\lambda) T} \omega(T,h) \leq 1/4$, we obtain:
\[
\begin{split}
&\| \chi_x u|_{W^u(x)} \|_{\overline{H}_h^{2N}(W^u(x))}^2  \\
& \qquad  \leq  C\left(\|u\|^2_{H^{-2N+d^\perp/2 + \gamma}} + \|u\|^2_{\Gamma,N,n}\right)  + \dfrac{1}{4} \|\chi_{\varphi_{-T} x} u|_{W^u(\varphi_{-T}x)}\|_{\overline{H}_h^{2N}(W^u(\varphi_{-T} x))}^2.
\end{split}
\]
Finally, the same argument by induction as in the proof of Corollary \ref{corollary:main} allows to conclude and proves \eqref{equation:bound-main}. To prove \eqref{equation:bound-main2}, it suffices to follow the proof of \eqref{equation:bound-main2-corollary}; however, this time, one has to use Lemma \ref{lemma:continuity-leafwise-integration} to pass to the limit as $n \to +\infty$.
\end{proof}

\subsection{Global smoothness} \label{ssection:pr}We now prove Corollaries \ref{corollary:utile}, \ref{corollary:smoothness} and \ref{corollary:support}. In this paragraph, the foliation $\mc{F}$ is always assumed to be expanded by the flow.

\subsubsection{Propagating pieces of the foliation leaves} We will need a preliminary observation. Given $x \in M$ and $V \subset L(x)$, an open set of the leaf of $x$, we let
\[
\Lambda_+(V) :=\bigcap_{T > 0}  \overline{\bigcup_{t \geq T} \varphi_t(V)}.
\]
Recall that $\mc{F}$ is minimal if each leaf is dense in $M$. The following holds:

\begin{lemma}
\label{lemma:matin}
If $\mc{F}$ is minimal, then $\Lambda_+(V) = M$.
\end{lemma}

\begin{proof}
We claim that for all $\eps > 0$, there exists $R > 0$ such that for all $x \in M$, $L_R(x)$ (the ball of radius $R > 0$ centered at $x$ in the leaf $L(x)$) is $\eps$-dense in $M$. This statement easily implies that $\Lambda_+(V)=M$ as the flow is expansive on the leaves. Given $x \in M$, define $R_\eps(x) := \inf \{R > 0 ~:~ L_R(x) \text{ is $\eps$-dense}\} > 0$. Note that $R_\eps(x) < +\infty$ as each leaf is assumed to be dense in $M$. By continuity of the foliation $\mc{F}$ with respect to $x$, one sees that the function $x \mapsto R_\eps(x)$ is upper semi-continuous on $M$. Therefore, it is bounded from above, that is $\sup_{x \in M} R_\eps(x) \leq R < +\infty$. This proves the claim. 
\end{proof}

For Anosov flows, the same result holds provided transitivity is assumed:

\begin{lemma}
\label{lemma:fino}
Suppose that $\varphi_t : M \to M$ is a transitive Anosov flow. Let $x \in M$ and $V \subset W^u(x)$ be an open subset. Then $\Lambda_+(V) = M$. 
\end{lemma}

\begin{proof}
By \cite[Theorem 1.8]{Plante-72}, under the assumption that the flow is transitive, only the following two possibilities can occur: (i) every strong stable/unstable leaf is dense in $M$; (ii) the flow is a suspension by a constant roof function. If (i) holds, then Lemma \ref{lemma:matin} can be applied. For a suspension flow by a constant roof function (case (ii)), the strong stable/unstable leaves are dense in each slice of the suspension, and since $\Lambda_+(V)$ is also $\varphi_t$-invariant one deduces that it equals $M$. %one can argue similarly to (i).
\end{proof}

We refer the reader to \cite{AvilaCrovisierWilkinson} and references therein for other instances where minimality of the strong unstable foliation is expected in the partially hyperbolic setting. 

\subsubsection{Global smoothness} We now prove Corollaries \ref{corollary:utile} and \ref{corollary:smoothness}:

\begin{proof}[Proof of Corollary \ref{corollary:utile}]
Let $u \in \mc{D}'_{\Gamma}(M,E)$ such that $(\X-\lambda) u = 0$. Suppose that $u|_V$ is smooth for some open subset $V \subset W^u(x_0)$. By Lemma \ref{lemma:matin}, any point $x \in M$ belongs to $\Lambda_+(V) := \cap_{T > 0} \overline{\cup_{t \geq T} \varphi_t(V)}$. We can thus apply Theorem \ref{theorem:main} for large values of $N \gg 1$ (depending on $\lambda$); this yields that for all $x \in M$, $u \in H^N(W^u_1(x),E)$ with uniformly controlled norm with respect to $x \in M$. As this holds for all $N \gg 1$, we obtain that $u \in C^\infty(W^u_1(x),E)$ with uniformly controlled norm with respect to $x \in M$.

Let $x_0 \in M$ and $\psi$ be a small bump function around $x_0$. Let $S \in C^\infty(M)$ be a smooth function such that $dS(x_0) \notin N^*\mc{F}$. We apply the Fubini formula \eqref{equation:fubini} with $\varphi = \psi e^{iS/h}$:
\[
(u, \psi e^{iS/h} \dd\vol) = \int_{\Sigma} (\mathbf{r}_{W^u(x)}(u), \chi e^{iS/h} \chi_x \mu_x) \dd \nu(x).
\]
Note that $\mathbf{r}_{W^u(x)}(u)$ is (uniformly) smooth on $W^u(x)$. As in the proof of Lemma \ref{lemma:leafwise-density}, writing $(\mathbf{r}_{W^u(x)}(u), \chi e^{iS/h} \chi_x \mu_x)$ as an integral over $W^u(x)$, one can then integrate by parts using leafwise smooth vector fields $Y_j$ which are tangent to the leaves. Observe that, as $dS(x_0) \neq 0$, there is a $j$ such that $dS(x_0)(Y_j) \neq 0$. This shows that $(u, \psi e^{iS/h} \dd\vol) = \mc{O}(h^\infty)$, so $\WF(u) \subset N^*\mc{F}$. However, $\WF(u) \subset \Gamma$ and $\Gamma$ is disjoint from $N^*\mc{F}$ by assumption. Therefore, $u \in C^\infty(M,E)$.
\end{proof}

\begin{proof}[Proof of Corollary \ref{corollary:smoothness}]
Same proof as Corollary \ref{corollary:utile} with $\Gamma = E_s^\perp$, $\mc{F}=W^u$, using Lemma \ref{lemma:fino}, and the absolute continuity of the unstable foliation in this case (see \cite[Theorem 2]{DeLaLlave-01} for instance).
\end{proof}

\subsubsection{Support} Finally, we turn to the proof of Corollary \ref{corollary:support}:

\begin{proof}[Proof of Corollary \ref{corollary:support}]
Let $u \in \mc{D}'_{E_s^\perp}(M,E)$ be a Pollicott-Ruelle resonant state associated with $\lambda$, that is $(\X-\lambda) u = 0$. Suppose that $u|_V = 0$ for some open subset $V \subset W^u(x)$. Then $u \equiv 0$ on the flow-out $\Omega_+(V) = \cup_{t \geq 0} \varphi_t(V)$. Following the arguments in the proof of Theorem \ref{theorem:main}, this easily implies that for all $x \in \Lambda_+(V)$, $u|_{W^u_1(x)} \equiv 0$. Applying the Fubini formula \eqref{equation:fubini} (the unstable foliation is absolutely continuous in this case, see \cite[Theorem 2]{DeLaLlave-01}), we find that $u \equiv 0$. 
\end{proof}

\bibliographystyle{alpha}
%\nocite{*}
\bibliography{biblio}

\end{document}